\documentclass[11pt]{article}

\usepackage[T1]{fontenc}
\usepackage{lmodern}
\usepackage{microtype}
\usepackage[margin=1in]{geometry}
\usepackage{amsmath,amssymb,amsthm,mathtools}
\usepackage{aliascnt}
\usepackage{booktabs,tabularx,array}
\usepackage{graphicx}
\usepackage{enumitem}
\usepackage[authoryear]{natbib}
\usepackage{xcolor}
\newcommand{\pRevision}[1]{{\color{blue}#1}}
\usepackage[colorlinks=true,linkcolor=blue!55!black,citecolor=blue!55!black,urlcolor=blue!55!black]{hyperref}
\usepackage{doi}
\usepackage{hyperxmp}
\usepackage[nameinlink,noabbrev,capitalize]{cleveref}

\allowdisplaybreaks
\newtheorem{theorem}{Theorem}[section]
\newaliascnt{lemma}{theorem}
\newtheorem{lemma}[lemma]{Lemma}
\aliascntresetthe{lemma}
\newaliascnt{proposition}{theorem}
\newtheorem{proposition}[proposition]{Proposition}
\aliascntresetthe{proposition}
\newaliascnt{corollary}{theorem}
\newtheorem{corollary}[corollary]{Corollary}
\aliascntresetthe{corollary}
\newaliascnt{remark}{theorem}
\newtheorem{remark}[remark]{Remark}
\aliascntresetthe{remark}
\crefname{theorem}{Theorem}{Theorems}
\crefname{lemma}{Lemma}{Lemmas}
\crefname{proposition}{Proposition}{Propositions}
\crefname{corollary}{Corollary}{Corollaries}
\crefname{remark}{Remark}{Remarks}

\newcommand{\R}{\mathbb R}
\newcommand{\E}{\mathbb E}
\newcommand{\Var}{\operatorname{Var}}
\newcommand{\Gap}{\operatorname{Gap}_{\mathrm{R}}}
\newcommand{\TV}{\operatorname{TV}}
\newcommand{\cE}{\mathcal E}
\newcommand{\cK}{\mathcal K}
\newcommand{\ip}[2]{\left\langle #1,#2\right\rangle}
\newcommand{\norm}[1]{\left\lVert #1\right\rVert}
\newcommand{\abs}[1]{\left\lvert #1\right\rvert}
\newcommand{\1}{\mathbf 1}
\newcommand{\df}{\mathrm{d}}

\title{\bf A spectral gap for Metropolis-adjusted Langevin algorithm with a uniformly randomized step size}
\author{Qian Qin \\ School of Statistics \\ University of Minnesota}
\date{}
\hypersetup{pdfauthor={Qian Qin}}
\ifdefined
\fi

\begin{document}
	\maketitle
	
	\begin{abstract}
		Let $\pi(\mathrm{d} x)\propto e^{-U(x)}\, \mathrm{d} x$ on $\R^d$, where
		$U$ is continuously differentiable and $m$-strongly convex with a globally
		$L$-Lipschitz gradient, $0<m\leq L<\infty$, and $\kappa=L/m$.
		Fixed-step Metropolis-adjusted Langevin algorithm (MALA) has known
		warm-start mixing-time upper bounds of order $\kappa\sqrt d$, up to logarithmic
		factors. A spectral-gap lower bound at the corresponding scale $(\kappa \sqrt{d})^{-1}$ would give an
		upper bound of order $\kappa \sqrt{d}$ on the Monte Carlo asymptotic variance relative to
		independent sampling, uniformly over all square-integrable functions. 
		However, when $\kappa$ is bounded away from~1, the
		best fixed-step spectral gap that can be guaranteed uniformly over this
		target class is at most of order $\max\{\log(\kappa d)/(\kappa d), e^{-cd}\}$ for some positive universal constant~$c$.
		We show that uniform randomization of the step size improves this
		worst-case guarantee. At each iteration, the algorithm draws
		$h\sim\operatorname{Unif}(0,H)$ and performs one ordinary MALA transition.
		Choosing $H$ of order
		$[L\sqrt{d(1+\log d+\log\kappa)}]^{-1}$ yields a right
		spectral-gap lower bound of order
		\[
		\frac{1}{\kappa\sqrt{d(1+\log d+\log\kappa)}},
		\]
		uniformly over the target class. 
		Thus, given $\kappa > 1$, for all square-integrable functions, the ratio of the Monte Carlo asymptotic variance relative to independence sampling has an upper bound of order $\sqrt{d \log d}$.
		This contrasts with fixed-MALA, where the ratio can be as bad as $d/\log d$ in thew worst-case scenario.
		The spectral gap also gives geometric convergence of
		the lazy kernel from every initial density in $L^2(\pi)$, central limit
		theorems, and nonstationary mean-square error bounds.
		The proof combines acceptance and flow estimates at different step-size
		scales using a Cheeger-type aggregation inequality. The scale used to
		control flow out of a set may depend on its stationary mass, and the
		aggregation inequality avoids the loss from first estimating the overall
		conductance and then applying the standard Cheeger inequality.

		This work was developed with substantial assistance from ChatGPT, which suggested the uniformly randomized-step approach, developed the principal proof arguments, and generated the simulation and Lean 4 code. The human author checked and verified the mathematical content and takes full responsibility for the results.
	\end{abstract}
	
	\section{Introduction}

	Markov chain Monte Carlo (MCMC) methods are sampling algorithms that can be used to estimate expectations
	under distributions from which independent sampling is difficult.
	The Metropolis-adjusted Langevin algorithm (MALA) uses the gradient of
	the log-density to construct proposals for this purpose.
	For a target probability distribution~$\pi$ on $\R^d$ with density
	proportional to $e^{-U}$, MALA
	proposes
	\[
	Y=x-h\nabla U(x)+\sqrt{2h}\,Z,
	\qquad Z\sim N_d(0,I_d),
	\]
	and applies the Metropolis--Hastings correction. The step size~$h$ controls
	both the distance of a proposed move and its probability of acceptance.
	The geometric ergodicity and high-dimensional scaling of MALA were studied
	by \citet{RobertsTweedie1996,RobertsRosenthal1998}. More recent work gives
	quantitative bounds on the number of iterations needed to approximate~$\pi$.
	In particular, when $U$ is $m$-strongly convex and has an $L$-Lipschitz
	gradient, \citet{WuSchmidlerChen2022}, strengthening
	\citet{ChewiEtAl2021}, establish a warm-start mixing-time upper bound for lazy MALA with
	$\kappa\sqrt d$ dependence, up to logarithmic factors in $d, \kappa$, the warmness parameter~$M$, and the desired total variation distance tolerance~$\epsilon$.
	Here, $\kappa =L/m$ is the condition number, and $M = \|\df \mu/\df \pi\|_{L^{\infty}(\pi)}$ where $\mu$ is the initial distribution.
	Their step size is of order
	$(L\sqrt d)^{-1}$ up to logarithmic factors that again depend on $d,k,M,\epsilon$. 
	Their upper bound is accompanied by lower bounds that identify
	the $\kappa\sqrt d$ dependence, up to logarithmic factors, in the
	warm-start mixing problem considered there.

	For Monte Carlo estimation, it is also crucial to control the dependence
	among successive samples; see, e.g., \cite{Geyer1992,Sokal1997,JonesHobert2001,JonesQin2022}. For a stationary chain with kernel $P$,
	the Monte Carlo asymptotic variance of a test function $f$ is
	\[
	\sigma_P^2(f)=\lim_{n\to\infty}
	n\Var\left(\frac1n\sum_{k=0}^{n-1}f(X_k)\right),
	\]
	when the limit exists. For a reversible kernel with a positive right,
	or Poincar\'e, spectral gap, denoted by $\Gap(P)$, the standard spectral
	formula gives
	\[
	\sup_{\substack{f\in L^2(\pi)\\\Var_\pi(f)>0}}
	\frac{\sigma_P^2(f)}{\Var_\pi(f)}
	=\frac{2}{\Gap(P)}-1;
	\]
	see \citet{KipnisVaradhan1986,RobertsRosenthal2008VarianceBounding}.
	The ratio on the left compares the long-run variance of an MCMC average
	with that of an average of independent draws from~$\pi$.
	A lower bound on the spectral gap controls this ratio for every
	square-integrable test function.
	
	This leads to the question of whether a MALA kernel admits a
	spectral-gap lower bound of order $(\kappa\sqrt d)^{-1}$, up to
	logarithmic factors, uniformly over the target class described above.
	For a fixed step size $h>0$, let $P_h$
	denote the MALA kernel. The desired bound would have the form
	\begin{equation}\label{eq:gap-ideal}
		\Gap(P_h)\gtrsim\frac{1}{\kappa\sqrt d},
	\end{equation}
	where $h$ may depend on $d,m,L$, and $\gtrsim$ here suppresses universal
	constants and logarithmic factors in $d, \kappa$. Besides a warm-start mixing-time upper bound matching the order in \cite{WuSchmidlerChen2022},
	this would give a uniform bound of order $\kappa\sqrt d$, up to
	logarithmic factors, on the asymptotic variance relative to independent
	sampling. 
	Prior to the current work, a fixed-step spectral-gap lower bound of order $(\kappa^2d)^{-1}$ is
	available with a step size of order $(\kappa Ld)^{-1}$
	\citep{LiuTong2026}.
	
	However, \eqref{eq:gap-ideal} cannot hold uniformly over the target
	class for a single fixed step size. Spectral-gap obstructions of this
	type were established in various regimes by
	\citet{ChewiEtAl2021,LeeShenTian2021,WuSchmidlerChen2022}.
	Using a modification of their hard-target constructions,
	\cref{prop:minimax-fixed-step-ceiling} gives an explicit worst-case
	spectral-gap upper bound for every choice of $h$.
	For each fixed $h > 0$, when $\kappa$ is bounded away from~1, the largest spectral gap that can be guaranteed uniformly over the target class is at most of order
	$\max\{\log(\kappa d)/(\kappa d), e^{-cd}\}$ for some universal constant $c > 0$. Thus, given $\kappa > 1$, no fixed choice of~$h$ gives a worst-case asymptotic
	variance ratio smaller than order $d/\log d$ as $d\to\infty$.

	We study the uniformly randomized-step kernel
	\[
	\overline P_H=\frac1H\int_0^H P_h\,\df h.
	\]
	At each iteration, the algorithm draws $h\sim\operatorname{Unif}(0,H)$
	independently and performs one ordinary MALA transition. It uses one
	proposal per iteration, has essentially the same computational cost as
	fixed-step MALA, and is exactly $\pi$-reversible. Under strong
	convexity and a globally Lipschitz gradient, \cref{thm:main} gives a
	positive spectral-gap lower bound for every $H>0$. In particular,
	\cref{cor:sqrt-d-endpoint} shows that choosing $H$ of order
	\[
	\frac{1}{L\sqrt{d(1+ \log d +\log\kappa)}}
	\]
	gives a gap of at least order
	\[
	\frac{1}{\kappa\sqrt{d(1+ \log d +\log\kappa)} }.
	\]
	\Cref{cor:asymptotic-variance} shows that the asymptotic variance $\sigma_{\overline P_H}^2(f)$ is at
	most a constant multiple of
	$$\kappa\sqrt{d(1+\log d+\log\kappa)}\,\Var_\pi(f)$$ for every
	$f\in L^2(\pi)$. For fixed $\kappa>1$, this is an upper bound of order
	$\sqrt{d\log d}$ on the variance ratio, whereas every fixed step size
	has a worst-case ratio of at least order $d/\log d$.
	\Cref{cor:variance-separation} states this comparison explicitly.
	The comparison is uniform over the target class and test functions;
	it does not assert that randomization reduces the variance for every
	individual target and function.
	
	The spectral gap further supplies the hypotheses of standard results for
	Monte Carlo error. Under stationarity, every $f\in L^2(\pi)$ satisfies
	a central limit theorem
	\citep{KipnisVaradhan1986,RobertsRosenthal2008VarianceBounding}.
	\Cref{cor:nonstationary-variance} gives a finite-sample mean-square
	error bound for the lazy chain when the initial density belongs to
	$L^2(\pi)$ and $f\in L^4(\pi)$. 
	This is derived using established error estimates for general Markov chains with spectral gaps
	\citep{Rudolf2012,Paulin2015}.

	A spectral gap also gives a warm-start mixing-time upper bound, but it provides this through a stronger geometric convergence result.
	For a lazy or positive semidefinite $\pi$-reversible kernel $P$ and
	an initial distribution $\mu$ with $\df\mu/\df\pi\in L^2(\pi)$,
	\begin{equation}\label{eq:TVbound}
		\begin{aligned}
			\norm{\mu P^n-\pi}_{\mathrm{TV}}
			&\leq\frac12\norm{\frac{\df(\mu P^n)}{\df\pi}-1}_{L^2(\pi)}\\
			&\leq\frac12\norm{\frac{\df\mu}{\df\pi}-1}_{L^2(\pi)}
			[1-\Gap(P)]^n.
		\end{aligned}
	\end{equation}
	See \citet{RobertsRosenthal1997Hybrid,KontoyiannisMeyn2012,Rudolf2012}.
	Applied to $(I+\overline P_H)/2$, this yields \cref{cor:mixing}, which gives a warm-start mixing-time upper bound of order $\kappa \sqrt{d}$, up to logarithmic factors.
	Unlike warm-start mixing time bounds in \cite{ChewiEtAl2021,WuSchmidlerChen2022}, the step-size parameter~$H$ does not depend on the initial density or the desired accuracy. 
	Thus, the spectral-gap lower bound certifies convergence all the way to~$\pi$, without the need to tune~$H$ in accordance with a prescribed tolerance.
	

	The proof of our spectral gap bound uses different step-size ranges to control flow out of
	different sets. This requires more than integrating known fixed-step
	spectral-gap lower bounds. Indeed, the usual inheritance inequality
	\citep{GrazziLivingstoneRiouDurand2026} gives
	\[
	\Gap(\overline P_H)\geq\frac1H\int_0^H\Gap(P_h)\,\df h,
	\]
	but substituting available uniform fixed-step lower bounds does not give
	the desired rate. Each component gap $\Gap(P_h)$ already takes an infimum over all
	test functions. Estimating these gaps separately therefore loses the
	possibility of using different components to control different test functions.
	
	We instead estimate the one-step flow out of a set $S \subset \mathbb{R}^d$, given by
	\[
	\mathcal J_{\cK_t}(S,S^c)=\int_S \cK_t(x,S^c)\,\pi(\df x),
	\]
	for the kernels $\cK_t=(2/t)\int_{t/2}^tP_h\,\df h$, choosing the step scale $t \in (0,H]$
	according to $\pi(S)$. For sets of moderate mass, relatively large steps
	have good acceptance outside an exceptional region small enough to be
	discarded. For smaller sets, we use stronger acceptance estimates at
	smaller scales.
	Isoperimetry converts these acceptance estimates into flow bounds.
	This way, for each set $S$ satisfying $0 < \pi(S) < 1/2$, we control the flow out of~$S$ driven by MALA for a step size range $h \in (t/2,t) \subset (0,H)$.
	
	Finally, a Cheeger-type aggregation lemma combines the flow bounds into a spectral-gap bound. This procedure avoids the loss incurred by first bounding the conductance of
	the mixture $\overline P_H$ and then applying the usual Cheeger inequality. The closest analytic predecessor of the aggregation lemma is the general
	Cheeger inequality of \citet{ChenWang2000}. We give a direct proof in
	\cref{lem:fractional} and a formulation in \cref{thm:aggregation} that can be readily applied to our flow estimates.
	
	As a byproduct of our analysis, we obtain an $L^p(\pi)$ bound on the rejection probability of fixed-step MALA.
	Let $r_h(x)$ denote the probability of rejecting a proposal from~$x$ with step size~$h$.
	In \cref{prop:stationary-rejection} it is shown that, for some universal positive constants $c, C$,
	\[
	p \geq 1, \; 0 < h < \frac{c}{L\sqrt{p(d+p)}} \quad \Longrightarrow \quad \|r_h\|_{L^p(\pi)} \leq C Lh \sqrt{p(d+p)}.
	\]
	This estimate does not require convexity of~$U$.
	It complements existing MALA rejection bounds in the literature \citep{ChewiEtAl2021,WuSchmidlerChen2022,ChenGatmiry2023}.
	Its closest predecessor is the moment estimate in Lemma~17 of \cite{ChewiEtAl2021}, which is stated under strong convexity of~$U$.

	An accompanying Lean~4 development \citep{qinMALALean2026} provides proof
	scripts for the main theorem and other principal conclusions, including
	\cref{thm:main,prop:minimax-fixed-step-ceiling,lem:fractional,thm:aggregation}.
	The scope of these scripts is described in \cref{sec:lean-verification}.
	
	\subsection*{Relation to existing work}
	
	We now make some additional remarks regarding existing work relevant to our study.
	
	Quantitative spectral-gap bounds for MALA exist in the literature, but are less common than
	mixing-time bounds. 
	The classical work of \citet{RobertsTweedie1996}
	gives qualitative criteria for geometric ergodicity.
	\citet{BouRabeeHairer2013} show that, for drifts that are not globally
	Lipschitz, MALA may fail to have a spectral gap even when the underlying
	Langevin diffusion does. In the strongly convex, gradient-Lipschitz
	setting, \citet[Theorem~2]{LeeShenTian2021} give a fixed-step gap upper
	bound of order $\log d/(\kappa d)$ under their restriction
	$\kappa=O(d^4)$;
	see also \cite{ChewiEtAl2021,WuSchmidlerChen2022}.
	As mentioned earlier, \citet{LiuTong2026} obtain a spectral-gap
	lower bound of order $(\kappa^2d)^{-1}$ with a step size of order
	$(\kappa Ld)^{-1}$.

	Mixing from initial distributions whose warmness grows exponentially
	with dimension has also been studied. For example, the Gaussian
	distribution $N_d(x_\star,L^{-1}I_d)$, where $x_\star$ is the minimizer
	of $U$, has warmness at most $\kappa^{d/2}$.
	From this initialization, \citet[Theorem~5]{ChenDwivediWainwrightYu2020}
	obtain an $L^2$ mixing-time bound of order
	$\max\{\kappa d,\kappa^{3/2}\sqrt d\}\log(d/\epsilon)$ for lazy MALA,
	with a step size independent of $\epsilon$.
	\Citet{LeeShenTian2020} obtain related guarantees using blocking
	conductance and a procedure for improving a constant-accuracy bound.
	Like in \cite{ChewiEtAl2021,WuSchmidlerChen2022}, the mixing-time upper bound herein works best when the warmness grows sub-exponentially with~$d$.
	\Citet{AltschulerChewi2024} provided a way to construct such warm starts using underdamped
	Langevin Monte Carlo before running MALA.

	Randomized tuning parameters have previously been studied as a way to
	improve robustness. \Citet{GrazziLivingstoneRiouDurand2026} study
	mixtures $\int P_h\,\nu(\df h)$, establish inheritance of spectral gaps
	and weak Poincar\'e inequalities, and show that randomization can reduce
	sensitivity to step-size misspecification. Related benefits are known
	for Hamiltonian methods
	\citep{BouRabeeSanzSerna2017,ApersGriblingSzilagyi2024}.
	Our result concerns a dimension- and condition-number-explicit global
	gap for uniform-random MALA over the full strongly convex,
	gradient-Lipschitz target class.
	
	Several established methods enter the acceptance analysis.
	\Citet{ChewiEtAl2021} use a projection characterization of Metropolis
	adjustment and compare the Euler proposal with exact Langevin diffusion.
	\Citet{WuSchmidlerChen2022} compare one leapfrog step with continuous
	Hamiltonian dynamics; the relation between MALA and Hamiltonian methods
	is also discussed in \citet{Neal2011,ChenGatmiry2023}.
	Earlier rejection estimates and Wasserstein convergence bounds for
	perturbations of Gaussian measures appear in \citet{Eberle2014} under
	higher-order regularity assumptions. \Citet{MangoubiVishnoi2019} analyze
	Hamiltonian energy error under third- and fourth-order regularity
	conditions, including for nonconvex targets. Pathwise comparisons
	between Metropolized integrators and diffusions also appear in
	\citet{BouRabeeVandenEijnden2010,BouRabeeHairer2013}.
	
	Our appendix follows the diffusion-comparison approach. The stationary
	increment estimate follows from the Lyons--Zheng forward--backward
	decomposition \citep[Section~1, especially equation~(1.7)]{LyonsZheng1988};
	see also \citet[Theorem~5.7.1]{FukushimaOshimaTakeda2011}.
	We give an elementary Langevin proof by subtracting the forward and
	reversed stochastic differential equations to cancel the drift.
	A Girsanov change of path law then compares the exact diffusion with the
	process whose drift is frozen at its initial state. Endpoint conditioning
	and the Metropolis--Hastings meet identity
	\citep{Tierney1998,BilleraDiaconis2001} give an $L^p(\pi)$ rejection bound.
	The close-coupling and isoperimetric arguments use methods familiar from
	Metropolis analyses
	\citep{DwivediEtAl2019,WuSchmidlerChen2022,AndrieuLeePowerWang2024,LiuTong2026}.
	The use of set mass to select a scale is related to average conductance,
	evolving sets, blocking conductance, and spectral profiles
	\citep{LovaszKannan1999,MorrisPeres2005,GoelMontenegroTetali2006,KannanLovaszMontenegro2006,Montenegro2007}.
	
	A complementary line of work studies alternative gradient-based
	samplers, including randomized-midpoint discretizations of underdamped
	Langevin dynamics \citep{ShenLee2019} and the path-space rejection
	methods of \citet{ChenChewiRakhlinZhang2026}. Under their assumptions,
	the latter achieve expected gradient-query complexity with
	$\kappa^{2/3}d^{1/3}$ dependence, up to logarithmic factors, for sampling
	to a prescribed R\'enyi-divergence tolerance. 
	These algorithms address the complexity of producing a sample at a
	specified accuracy. Our result concerns the convergence and Monte Carlo
	variance of a randomized-step MALA kernel, and is not a claim of
	optimal query complexity among all gradient-based samplers.
	
	\subsection*{Organization}
	The rest of the article is organized as follows.
	\Cref{sec:main} states the main bound, the fixed-step obstruction, and
	the consequences for mixing, asymptotic variance, and nonstationary
	Monte Carlo error, including the variance comparison in
	\cref{cor:variance-separation}.
	\Cref{sec:ingredients} develops the ingredients of the proof, which is
	completed in \cref{sec:main-proof}.
	\Cref{sec:lean-verification} describes the Lean development, and
	\cref{sec:simulations} gives numerical comparisons of fixed-step and
	randomized-step MALA. Technical proofs are collected in the Appendices.

	\subsection*{Generative AI disclosure}
	
	This project began when the author asked ChatGPT 5.6 Sol (OpenAI) whether MALA admits a useful spectral-gap bound at the $d^{-1/2}$ step-size scale. ChatGPT suggested studying a uniformly randomized-step version of MALA—an algorithmic variant that had previously appeared in the literature—and developed the principal proof strategy and initial versions of the major mathematical arguments in this article. It also assisted in drafting and revising the manuscript and generated all code used for the numerical experiments and the Lean 4 formal verification. The human author checked the relevant literature, scrutinized and revised the AI-generated arguments, verified the mathematical claims and computations, inspected the simulation methodology and results, and confirmed that the Lean development was kernel-checked without proof placeholders or problem-specific axioms. The human author takes full responsibility for the content and correctness of the article.

	\section{Setup and Main Results} \label{sec:main}
	
	Assume throughout, except where explicitly stated otherwise, that \(d\) is a
	positive integer and that \(U:\R^d\to\R\) is continuously differentiable,
	\(m\)-strongly convex, and has a globally \(L\)-Lipschitz gradient, where
	\(0<m\leq L<\infty\). Explicitly, for every \(x,y\in\R^d\),
	\begin{equation}\label{eq:first-order-assumptions}
		\begin{aligned}
			&U(y)\geq U(x)+\ip{\nabla U(x)}{y-x}+\frac m2\norm{y-x}^2,\\
			&\norm{\nabla U(y)-\nabla U(x)}\leq L\norm{y-x}.
		\end{aligned}
	\end{equation}
	Strong convexity implies that \(e^{-U}\) is integrable.
	Let \(\kappa:=L/m\geq1\), and define the target distribution
	\[
	\pi(\df x)=\frac{e^{-U(x)}}{\int_{\R^d}e^{-U(u)}\,\df u}\,\df x.
	\]
	For \(h>0\), let \(Q_h(x,\cdot)\) be the Gaussian law
	\[
	N_d \bigl(x-h\nabla U(x),2hI_d\bigr).
	\]
	The fixed-step MALA kernel is
	\[
	P_h(x,\df y)=\alpha_h(x,y) \, Q_h(x,\df y)
	+\left[1-\int\alpha_h(x,z) \, Q_h(x,\df z)\right]\delta_x(\df y),
	\]
	where
	\[
	\alpha_h(x,y)
	=1\wedge
	\frac{e^{-U(y)} \, q_h(y,x)}{e^{-U(x)} \, q_h(x,y)}
	\]
	and \(q_h\) is the density of \(Q_h\).  The usual
	Metropolis--Hastings argument makes \(P_h\) \(\pi\)-reversible
	\citep{Tierney1998}; consequently, so is the randomized-step MALA kernel
	\[
	\overline P_H(x,\df y)=\frac{1}{H}\int_0^H P_h(x, \df y)\,\df h.
	\]
	Thus $H$ is fixed, while a new step size~$h$ is drawn independently at each
	iteration. The Metropolis--Hastings correction uses the proposal density
	for the step size drawn in that iteration.
	
	For probability measures \(\mu,\nu\) on a measurable space $(\Omega,\mathcal{F})$, write
	\[
	\TV(\mu,\nu)=\sup_{A \in \mathcal{F}}\abs{\mu(A)-\nu(A)}.
	\]
	For a probability measure $\varpi$, let $L^2(\varpi)$ be the space of
	measurable functions $f:\Omega\to\R$ such that
	$\int f^2\,\df\varpi<\infty$.  For a $\varpi$-reversible kernel $P$ and
	$f\in L^2(\varpi)$, define the Dirichlet form
	\[
	\cE_P(f,f)
	=\frac12\int[f(x)-f(y)]^2\,\varpi(\df x)P(x,\df y).\]
	The right, or Poincar\'e, spectral gap is defined as  
	\begin{equation} \label{eq:gap}
		\Gap(P)=\inf_{\substack{f\in L^2(\varpi)\\ \Var_\varpi(f)>0}}
		\frac{\cE_P(f,f)}{\Var_\varpi(f)}.
	\end{equation}
	This quantity lies in $[0,2]$. For a lazy or positive semidefinite kernel,
	it gives the geometric convergence bound in \eqref{eq:TVbound}.
	For a general reversible kernel, a positive right gap also controls
	asymptotic variance; laziness is not needed for this conclusion.
	
	%
	
	
	The following theorem gives a spectral-gap lower bound for the randomized kernel.
	The constants are uniform over all potentials satisfying
	\eqref{eq:first-order-assumptions}. 
	
	\begin{theorem} \label{thm:main}
		There is a universal constant $c_0\in(0,\infty)$ such that the following holds.  Define
		\begin{equation*}
			p_\star=1+\log d+\log\kappa.
		\end{equation*}
		Then, for every \(H \in (0,\infty)\),
		\begin{equation}\label{eq:master-gap}
			\Gap(\overline P_H)
			\geq
			c_0\frac{m}{H}
			\min\left\{
			H,
			\frac1L
			\max\left\{
			\frac{1}{\sqrt{p_\star(d+p_\star)}},
			\frac1d
			\right\}
			\right\}^{2}.
		\end{equation}
		The right spectral gap of the lazy kernel $(I+\overline P_H)/2$ is half of that of $\overline P_H$.
	\end{theorem}
	
	The bound is positive for every $H>0$. As a function of~$H$, it grows
	linearly until $H$ reaches the second entry of the minimum in
	\eqref{eq:master-gap}, and then decreases as $H^{-1}$. The next corollary
	gives the dimension dependence for two useful choices of~$H$.
	
	\begin{corollary}\label{cor:sqrt-d-endpoint}
		Under the notations of Theorem~\ref{thm:main},
		let \(c \in (0,\infty)\), and let \(H=c/(L\sqrt d)\).
		Then
		\[
		\Gap(\overline P_H) \geq \frac{c_0}{\kappa\sqrt{d}}
		\min \left\{ c, \frac1c
		\max \left\{ \frac{d}{p_\star(d+p_\star)}, \frac{1}{d} \right\} \right\},
		\]
		where $c_0 \in (0,\infty)$ is a universal constant.
		In particular, since $\min\{p_\star(d+p_\star)/d,d\}\leq2p_\star$,
		\[
		\Gap(\overline P_H) \geq \frac{c_0}{\kappa\sqrt{d}}
		\min\left\{ c, \frac1{cp_\star} \right\},
		\]
		after adjusting the universal constant $c_0$.
		Choosing \(c=c'/\sqrt{p_{\star}}\) for some $c' \in (0,\infty)$ gives
		\[
		\Gap(\overline P_H)
		\geq
		\frac{c_0}{\kappa\sqrt{d\,p_\star}}
		\min\left\{c',\frac1{c'}\right\}.
		\]
	\end{corollary}
		%
		%

	Taking $c'$ to be a fixed positive constant gives a gap lower bound
	of order $(\kappa\sqrt{d\,p_\star})^{-1}$.
	With $c$ fixed instead, the larger endpoint $H=c/(L\sqrt d)$ gives a
	lower bound of order $(\kappa\sqrt d\,p_\star)^{-1}$.
	Thus the smaller choice of~$H$ incurs a square-root logarithmic loss,
	and the larger choice incurs a full logarithmic loss.
	Both choices are independent of the desired accuracy.
	The following result, proved in \cref{app:fixed-step-obstruction},
	shows why a comparable uniform gap cannot be obtained by choosing a
	single fixed step size.
	
	\begin{proposition}
		\label{prop:minimax-fixed-step-ceiling}
		Fix $\kappa_0>1$.  There are a universal constant $c \in (0,\infty)$ and a
		constant $C \in (0,\infty)$ depending only on $\kappa_0$ such that the
		following holds.  Let $d\geq2$, $0<m<L$, and
		$\kappa=L/m\geq\kappa_0$.  Then
		\begin{equation}
			\label{eq:minimax-fixed-step-ceiling}
			\sup_{h>0}
			\inf_{\substack{U\in C^\infty(\R^d)\\
					mI_d\preceq\nabla^2U(x)\preceq LI_d,\ x\in\R^d}}
			\Gap(P_{U,h})
			\leq
			C\max\left\{
			\frac{\log(\kappa d)}{\kappa d},
			e^{-cd}
			\right\}.
		\end{equation}
		Here $P_{U,h}$ denotes the MALA kernel with step size~$h$ and
		target density proportional to $e^{-U}$.
	\end{proposition}
	
	\begin{remark} \label{rem:minimax-fixed-step-ceiling}
		In \cref{app:small-fixed-step-gap}, it is shown that, when $U$ satisfies \eqref{eq:first-order-assumptions}, it holds that $\Gap(P_{1/(2Ld)}) \geq c (\kappa d)^{-1}$ for some universal constant $c \in (0,\infty)$, so the upper bound in \cref{prop:minimax-fixed-step-ceiling} can be matched with a lower bound.
	\end{remark}
	
	The order of the supremum and infimum in
	\eqref{eq:minimax-fixed-step-ceiling} is important: for each choice
	of~$h$, there is a hard potential, which may depend on~$h$.
	The result concerns selecting a step size uniformly over the target
	class. For each fixed $\kappa>1$, it bounds the best such gap by a
	constant multiple of $(\log d)/d$ as $d\to\infty$, whereas
	\cref{cor:sqrt-d-endpoint} gives the randomized kernel a uniform lower
	bound of order $(d\log d)^{-1/2}$. This separation is polynomial in
	dimension up to logarithmic factors. The hard-target construction builds
	on \citet{ChewiEtAl2021,LeeShenTian2021,WuSchmidlerChen2022}; our
	modification gives the explicit dependence in
	\eqref{eq:minimax-fixed-step-ceiling} for all $\kappa\geq\kappa_0>1$.
	

	\cref{cor:sqrt-d-endpoint} and \eqref{eq:TVbound} give a mixing-time bound.
	Let \(\mu\ll\pi\) be a probability distribution on \(\R^d\) with
	\(\df\mu/\df\pi\in L^2(\pi)\). For a Markov kernel \(P\) with
	stationary distribution \(\pi\) and \(\epsilon>0\), define
	\[
	n_{\mathrm{mix}}(P,\mu,\epsilon)
	=\inf\left\{n\in\{0,1,\ldots\}:
	\norm{\mu P^n-\pi}_{\mathrm{TV}}\leq\epsilon\right\},
	\qquad \inf\varnothing=\infty.
	\]
	Then the following holds.
	
	\begin{corollary} \label{cor:mixing}
		Under the notations of Theorem~\ref{thm:main}, let \(H=c/(L\sqrt{d\,p_\star})\), where \(c \in (0,\infty)\), and put
		\[
		\ell_{\mu,\epsilon}
		=\log\max\left\{1,
		\frac{1}{2\epsilon}
		\norm{\frac{\df\mu}{\df\pi}-1}_{L^2(\pi)}\right\}.
		\]
		Then
		\[
		\begin{aligned}
			n_{\mathrm{mix}}\left(\frac{I+\overline P_H}{2},\mu,\epsilon\right)
			&\leq
			\left\lceil\frac{2\ell_{\mu,\epsilon}}
			{\Gap(\overline P_H)}\right\rceil
			\leq
			\left\lceil
			C\ell_{\mu,\epsilon}\,\kappa\sqrt {d p_{\star} } \right\rceil,
		\end{aligned}
		\]
		where \(C \in (0,\infty)\) is $\max\{c,c^{-1}\}$ multiplied by a universal constant.
	\end{corollary}
	\begin{proof}
		Since \(0<\Gap(\overline P_H)/2\leq1\), the inequality
		\(1-x\leq e^{-x}\) for \(x\in[0,1]\), together with
		\eqref{eq:TVbound} and \cref{cor:sqrt-d-endpoint} gives the desired mixing time bound.
		The logarithm defining
		\(\ell_{\mu,\epsilon}\) is finite and nonnegative, including when
		\(\mu=\pi\). If \(\ell_{\mu,\epsilon}=0\), the initial total-variation
		error is already at most \(\epsilon\), so the mixing time is zero.
		In the edge case \(\Gap(\overline P_H)/2=1\),
		\eqref{eq:TVbound} shows that the chain has law \(\pi\) after one step
		from \(\mu\); the same bound remains valid.
	\end{proof}

	The bound in \cref{cor:mixing} holds for every initial density in
	$L^2(\pi)$, using the same $H$ and the same geometric rate. A bounded
	density ratio is not required. 
	
	For comparison, \citet[Theorem~1]{WuSchmidlerChen2022} use a fixed
	step size of order
	\[
	[L\sqrt d\,\log^2(\max\{\kappa,d,M/\epsilon,c\})]^{-1}
	\]
	for an $M$-warm start (i.e., $\df \mu/\df \pi \leq M$), where $c>0$ is universal. Their acceptance estimates
	control the chain outside an exceptional region whose allowed size
	depends on $M$ and $\epsilon$. Decreasing this size can require a
	smaller step. Thus that theorem gives a family of mixing guarantees
	with tolerance-dependent tuning, whereas \cref{cor:mixing} follows
	from a geometric $L^2$ bound for one kernel. This is a distinction
	between the guarantees, and does not mean that a fixed-step MALA chain
	must be retuned to continue converging.

	A second consequence concerns Monte Carlo averages. For a stationary
	chain $(X_k)_{k\geq0}$ with kernel~$P$, write
	\[
	\sigma_P^2(f)=\lim_{n\to\infty}
	n\Var\left(\frac1n\sum_{k=0}^{n-1}f(X_k)\right)
	\]
	when this limit exists, allowing the value $+\infty$.
	The next corollary gives a finite bound and a central limit theorem.
	It is a consequence of the
	standard spectral formula for asymptotic variance
	\citep{KipnisVaradhan1986,Paulin2015}. Its role is to make the
	Monte Carlo implication of the spectral gap explicit.

	\begin{corollary}\label{cor:asymptotic-variance}
		Under the notation of \cref{thm:main}, let
		$H=c/(L\sqrt{d\,p_\star})$, where $c \in (0,\infty)$.
		For every real-valued measurable function $f\in L^2(\pi)$,
		the stationary asymptotic variance exists and satisfies
		\begin{equation}\label{eq:asymptotic-variance}
			\sigma_{\overline P_H}^2(f)
			\leq\left(\frac{2}{\Gap(\overline P_H)}-1\right)
			\Var_\pi(f)
			\leq C\kappa\sqrt{d\,p_\star}\,\Var_\pi(f),
		\end{equation}
		where $C\in(0,\infty)$ is $\max\{c,c^{-1}\}$ multiplied by a universal constant.
		Moreover, if $(X_k)_{k\geq0}$ has kernel $\overline P_H$ and
		an arbitrary initial probability distribution $\mu$ on $\R^d$, then
		\[
		\frac1{\sqrt n}\sum_{k=0}^{n-1}
		\left[f(X_k)-\int f\,\df\pi\right]
		\xrightarrow{\mathrm d}
		N\bigl(0,\sigma_{\overline P_H}^2(f)\bigr).
		\]
		For the half-lazy kernel $(I+\overline P_H)/2$, the corresponding
		statements hold with its own stationary asymptotic variance:
		\[
		\sigma_{(I+\overline P_H)/2}^2(f)
		\leq\left(\frac{4}{\Gap(\overline P_H)}-1\right)
		\Var_\pi(f)
		\leq 2C\kappa\sqrt{d\,p_\star}\,\Var_\pi(f).
		\]
	\end{corollary}
	
	\begin{proof}
		Let $f_0=f-\int f\,\df\pi$.
		Regard $\overline P_H$ as a self-adjoint contraction on $L^2(\pi)$.
		On the mean-zero subspace $\{g \in L^2(\pi): \int g\df \pi = 0\}$, the definition of the spectral gap~\eqref{eq:gap}
		gives $I- \overline P_H \geq\Gap(\overline P_H)I>0$. Thus its inverse exists there and
		\[
		0\leq(I-\overline P_H)^{-1}\leq\Gap(\overline P_H)^{-1}I.
		\]
		The spectral formula for stationary asymptotic variance
		\citep{KipnisVaradhan1986,Paulin2015} gives
		\[
		\begin{aligned}
			\sigma_{\overline P_H}^2(f)
			&=2\ip{f_0}{(I-\overline P_H)^{-1}f_0}_{L^2(\pi)}
			-\norm{f_0}_{L^2(\pi)}^2\\
			&\leq\left(\frac2{\Gap(\overline P_H)}-1\right)\Var_\pi(f).
		\end{aligned}
		\]
		\Cref{cor:sqrt-d-endpoint} gives the final inequality in
		\eqref{eq:asymptotic-variance}.
		
		We next verify Harris ergodicity. Since $q_h(x,y)>0$ and
		$\alpha_h(x,y)>0$ for every $h>0$ and $x,y\in\R^d$,
		$
		\overline P_H(x,A)>0$ whenever $\pi(A)>0$.
		Hence $\overline P_H$ is $\pi$-irreducible. It is also aperiodic:
		a period greater than one would require a set $D$ with
		$\pi(D)>0$ and $\overline P_H(x,D)=0$ for $\pi$-almost every $x\in D$,
		contradicting the preceding inequality.
		
		Let $A$ be measurable with $\pi(A)=1$. Since $\pi$ has an
		everywhere positive density, every accepted proposal belongs to
		$A$ almost surely. Write $r(x)=\overline P_H(x,\{x\})<1$ for the rejection
		probability. For $x\notin A$, the chain remains at $x$ until its
		first accepted move, and therefore
		\[
		\mathbb P_x(X_1\notin A,\ldots,X_n\notin A)=r(x)^n
		\longrightarrow0.
		\]
		For $x\in A$, we have $\overline P_H(x,A)=1$. Thus every set of full
		$\pi$-measure is reached almost surely from every state.
		By \citet[Theorem~6(v)]{RobertsRosenthal2006Harris}, $\overline P_H$ is
		Harris recurrent. Its invariant probability distribution is
		$\pi$, so it is positive Harris recurrent and aperiodic.
		
		Under stationarity, the finite variance bound and reversibility
		give the central limit theorem by
		\citet[Theorem~7]{RobertsRosenthal2008VarianceBounding};
		see also \citet{KipnisVaradhan1986}.
		For a Harris ergodic chain, a central limit theorem under
		stationarity holds from every initial distribution with the
		same limiting variance; see \citet[Remark~6]{Jones2004CLT}.
		This proves the assertion for arbitrary $\mu$.
		
		Finally, $(I+\overline P_H)/2$ is reversible and has right spectral gap
		$\Gap(\overline P_H)/2$. It has the same properties used above: its
		accepted moves have a positive density, and its holding
		probability is $[1+r(x)]/2<1$ at every state. Applying the
		same arguments proves the half-lazy assertions.
	\end{proof}

	For independent samples, the asymptotic variance is $\Var_\pi(f)$.
	Thus \cref{cor:asymptotic-variance} bounds the variance inflation
	uniformly over the target class and all square-integrable observables.
	The same value of~$H$ is used for every observable.

	The next corollary compares the worst-case asymptotic variance for $P_h$ and $\overline P_H$.

	\begin{corollary}\label{cor:variance-separation}
		Fix $\kappa_0>1$. There are universal constants $a,C>0$ and a
		constant $c>0$ depending only on $\kappa_0$ such that the following
		holds. Let $d\geq2$, $0<m<L$, and $\kappa=L/m\geq\kappa_0$, and
		write $p_\star=1+\log d+\log\kappa$.
		\begin{enumerate}[label=(\roman*)]
			\item For $H=1/(L\sqrt{d\,p_\star})$, every potential $U$ satisfying
			\eqref{eq:first-order-assumptions}, and every $f\in L^2(\pi)$,
			\[
			\sigma_{\overline P_H}^2(f)
			\leq C\kappa\sqrt{d\,p_\star}\,\Var_\pi(f).
			\]
			
			\item For every $h>0$, there are a smooth potential $U$ satisfying
			\eqref{eq:first-order-assumptions} and a function $f\in L^2(\pi_U)$
			such that
			\[
			\Var_{\pi_U}(f)=1,
			\qquad
			\sigma_{P_{U,h}}^2(f)
			\geq c\min\left\{
			\frac{\kappa d}{\log(\kappa d)},\ e^{ad}
			\right\}.
			\]
		\end{enumerate}
		Here $\pi_U$ is the target associated with $U$. The asymptotic
		variance in (ii) is allowed to be infinite.
	\end{corollary}
	
	\begin{proof}
		Assertion (i) follows directly from \cref{cor:asymptotic-variance}.
		
		For (ii), recall the spectral identity
		\begin{equation}\label{eq:worst-case-variance-gap}
			\sup_{\substack{g\in L^2(\pi)\\\Var_\pi(g)>0}}
			\frac{\sigma_P^2(g)}{\Var_\pi(g)}
			=\frac{2}{\Gap(P)}-1
		\end{equation}
		for a reversible kernel~$P$ with a unique invariant distribution~$\pi$,
		where the right side is interpreted as $+\infty$ when
		$\Gap(P)=0$; see
		\citet{KipnisVaradhan1986,RobertsRosenthal2008VarianceBounding}.
		
		Fix $h>0$. By \cref{prop:minimax-fixed-step-ceiling}, there are a
		universal constant $a>0$ and a constant $C_1>0$ depending only on
		$\kappa_0$ such that one can choose a smooth admissible potential
		$U$ with
		\[
		\Gap(P_{U,h})
		\leq 2C_1\max\left\{
		\frac{\log(\kappa d)}{\kappa d},\ e^{-ad}
		\right\}.
		\]
		The factor $2$ allows for the possibility that the infimum over
		$U$ in that proposition is not attained.
		
		We also have $\Gap(P_{U,h})\leq1$. Taking $f$ in \eqref{eq:gap} to be the indicator of a measurable set $A$ satisfying
		$0<\pi_U(A)<1$ gives
		\[
		\Gap(P_{U,h})
		\leq
		\frac{\int_A P_{U,h}(x,A^c)\,\pi_U(\df x)}
		{\pi_U(A)[1-\pi_U(A)]}
		\leq\frac1{1-\pi_U(A)}.
		\]
		Since $\pi_U$ has a density, such sets can have arbitrarily small
		positive probability. Letting $\pi_U(A)$ decrease to zero proves
		the claim.
		
		Consequently, \eqref{eq:worst-case-variance-gap} yields
		\[
		\begin{aligned}
			\sup_{\substack{g\in L^2(\pi_U)\\\Var_{\pi_U}(g)>0}}
			\frac{\sigma_{P_{U,h}}^2(g)}{\Var_{\pi_U}(g)}
			&=\frac{2}{\Gap(P_{U,h})}-1\\
			&\geq\frac1{\Gap(P_{U,h})}\\
			&\geq\frac1{2C_1}
			\min\left\{
			\frac{\kappa d}{\log(\kappa d)},\ e^{ad}
			\right\}.
		\end{aligned}
		\]
		By the definition of the supremum, there is a function~$f$ with
		variance ratio at least
		\[
		\frac1{4C_1}\min\left\{
		\frac{\kappa d}{\log(\kappa d)},\ e^{ad}
		\right\}.
		\]
		Rescaling this function to have stationary variance one proves
		(ii), with $c=1/(4C_1)$.
	\end{proof}
	
	For each fixed $\kappa>1$, the upper bound in (i) is of order
	$\sqrt{d\log d}$, whereas the lower bound in (ii) is of order
	$d/\log d$ as $d\to\infty$. The target and test function in (ii)
	may depend on $h$. Thus this is a comparison of uniform
	worst-case guarantees, rather than a comparison for every fixed
	target and test function.

	The next corollary extends the variance guarantee to a bound on the
	finite-sample mean-square error for the lazy randomized-step MALA. We use the half-lazy kernel
	so that its right gap is also its absolute spectral gap.
	
	\begin{corollary}\label{cor:nonstationary-variance}
		Under the notation of \cref{thm:main}, let
		$H=c/(L\sqrt{d\,p_\star})$, where $c>0$.
		Let $(X_k)_{k\geq0}$ have kernel
		$(I+\overline P_H)/2$ and initial distribution
		$\mu\ll\pi$, with $\df\mu/\df\pi\in L^2(\pi)$.
		For $f\in L^4(\pi)$, define
		\[
		f_0=f-\int f\,\df\pi,
		\qquad
		\overline f_n=\frac1n\sum_{k=0}^{n-1}f(X_k).
		\]
		Then, there is a universal constant $c_0 \in (0,\infty)$ such that, for every integer $n\geq1$,
		\begin{equation}\label{eq:nonstationary-mse}
			\begin{aligned}
				\E_\mu\left[
				\left(\overline f_n-\int f\,\df\pi\right)^2
				\right]
				\leq&
				\frac{4/\Gap(\overline P_H)-1}{n} \Var_{\pi}(f)+
				\frac{512}{n^2\Gap(\overline P_H)^2}
				\norm{\frac{\df\mu}{\df\pi}-1}_{L^2(\pi)}
				\norm{f_0}_{L^4(\pi)}^2 \\
				\leq& \frac{4 \max\{c,c^{-1}\} \kappa \sqrt{d p_{\star}}}{c_0 n} \Var_{\pi}(f) + \\
				&
				\frac{512 \max\{c^2,c^{-2}\} \kappa^2 d p_{\star} }{c_0^2 n^2 }
				\norm{\frac{\df\mu}{\df\pi}-1}_{L^2(\pi)}
				\norm{f_0}_{L^4(\pi)}^2
			\end{aligned}
		\end{equation}
	\end{corollary}
	
	\begin{proof}
		The kernel $(I+\overline P_H)/2$ is reversible and
		positive semidefinite, with operator norm
		\[
		1-\frac{\Gap(\overline P_H)}2<1
		\]
		on the mean-zero subspace of $L^2(\pi)$.
		Apply \citet[Theorem~3.41]{Rudolf2012} with
		$p=4$, no burn-in, and the centered function $f_0$.
		This gives
		\begin{equation} \label{ine:rudolf}
			\begin{aligned}
				\E_\mu\left[
				\left(\overline f_n-\int f\,\df\pi\right)^2
				\right]
				&\leq
				\E_\pi\left[
				\left(\overline f_n-\int f\,\df\pi\right)^2
				\right]\\
				&\quad+
				\frac{512}{n^2\Gap(\overline P_H)^2}
				\norm{\frac{\df\mu}{\df\pi}-1}_{L^2(\pi)}
				\norm{f_0}_{L^4(\pi)}^2.
			\end{aligned}
		\end{equation}
		For the stationary chain,
		\[
		0\leq
		\ip{f_0}{
			\left(\frac{I+\overline P_H}{2}\right)^k f_0
		}_{L^2(\pi)}
		\leq
		\left(1-\frac{\Gap(\overline P_H)}2\right)^k
		\Var_\pi(f).
		\]
		Consequently,
		\[
		\begin{aligned}
			\E_\pi\left[
			\left(\overline f_n-\int f\,\df\pi\right)^2
			\right]
			&=
			\frac1{n^2}\left[
			n\Var_\pi(f)+
			2\sum_{k=1}^{n-1}(n-k)
			\ip{f_0}{
				\left(\frac{I+\overline P_H}{2}\right)^k f_0
			}_{L^2(\pi)}
			\right]\\
			&\leq
			\frac{\Var_\pi(f)}n
			\left[
			1+2\sum_{k=1}^{\infty}
			\left(1-\frac{\Gap(\overline P_H)}2\right)^k
			\right]\\
			&=
			\frac{4/\Gap(\overline P_H)-1}{n}\Var_\pi(f).
		\end{aligned}
		\]
		This along with \eqref{ine:rudolf} proves the first inequality in \eqref{eq:nonstationary-mse}.
		The final inequality follows from \cref{cor:sqrt-d-endpoint}.
		
	\end{proof}

	\section{Ingredients of the Proof of \cref{thm:main}} \label{sec:ingredients}
	
	We now describe the main ingredients of the proof of \cref{thm:main}.

	For \(t>0\), define the MALA kernel averaged over the upper dyadic half of
	\((0,t)\) by
	\begin{equation*}
		\cK_t=\frac2t\int_{t/2}^tP_h\,\df h.
	\end{equation*}
	The uniform mixture $\overline P_H$ contains this component with weight \(t/(2H)\)
	whenever \(t\leq H\).
	The following elementary relation holds.
	
	\begin{lemma} \label{lem:Kt}
		Let \(0<t_J<\cdots<t_0\leq H\), and suppose that the intervals
		$(t_j/2,t_j)$, $j\in\{0,\dots,J\}$, are pairwise disjoint.
		Then, for $f \in L^2(\pi)$,
		\[
		\cE_{\overline P_H}(f,f) \geq \sum_{j=0}^J \frac{t_j}{2H} \cE_{\cK_{t_j}}(f,f).
		\]
	\end{lemma}
	
	\begin{proof}
		By nonnegativity and Tonelli's theorem,
		\[
		\cE_{\overline P_H}(f,f)
		=\frac1H\int_0^H\cE_{P_h}(f,f)\,\df h
		\geq\sum_{j=0}^J\frac1H
		\int_{t_j/2}^{t_j}\cE_{P_h}(f,f)\,\df h.
		\]
		The definition of $\cK_{t_j}$ gives the claim.
	\end{proof}
	
	We will analyze $\cK_t$ at different scale parameters~$t$, and aggregate results along an appropriate sequence of scale levels in a way that is more sophisticated than \cref{lem:Kt}.
	
	To be more specific, in \cref{prop:stationary-rejection}, we bound the fixed-step MALA rejection rate, and use
	this bound to establish a close-coupling, or one-step overlap, result for $\cK_t$ in \cref{prop:overlap}.
	The curvature assumption on~$U$ yields the isoperimetric inequality in
	\cref{prop:separated}.  We then combine \cref{prop:overlap,prop:separated} in
	\cref{prop:flow} to bound the one-step flow for $\cK_t$ out of a given set~$S$.
	Importantly, for each set~$S$, the flow out of~$S$ is controlled for one particular $\cK_t$ where the scale parameter~$t$ depends on $\pi(S)$.
	The general component-aggregation bound in
	\cref{thm:aggregation} combines the one-step flows at various scales, yielding \cref{thm:main}.
	
	\subsection{Close coupling}
	
	The next proposition, proved in \cref{app:rejection-overlap}, shows that
	$\cK_t(x,\cdot)$ and $\cK_t(y,\cdot)$ are close in total variation whenever
	$\norm{x-y}\leq\sqrt{t}/16$.  For the larger admissible values of $t$, this
	conclusion is guaranteed only when $x$ and $y$ lie outside a small exceptional
	set.
	
	\begin{proposition} \label{prop:overlap}
		There are universal constants \(c_{\mathrm r},C_{\mathrm r}>0\) such
		that, for every real \(p\geq 1\) and every $t$ satisfying
		\begin{equation}\label{eq:overlap-step-condition}
			0<t\leq
			\frac{c_{\mathrm r}}{L\sqrt{p(d+p)}},
		\end{equation}
		there is a measurable set \(G_{t,p}\subseteq\R^d\) satisfying 
		\begin{equation}\label{eq:overlap-exceptional}
			\pi(G_{t,p}^c)
			\leq
			\left[C_{\mathrm r}Lt\sqrt{p(d+p)}\right]^p,
		\end{equation}
		and 
		\begin{equation*}
			x,y\in G_{t,p},\quad
			\norm{x-y}\leq\frac{\sqrt{t}}{16}
			\quad\Longrightarrow\quad
			\TV(\cK_t(x,\cdot),\cK_t(y,\cdot))\leq\frac34.
		\end{equation*}
		Moreover, if $t' \in (0,1/(2Ld)]$, then
		\[
		x,y\in \R^d,\quad
		\norm{x-y}\leq\frac{\sqrt{ t'}}{16}
		\quad\Longrightarrow\quad
		\TV(\cK_{t'}(x,\cdot),\cK_{t'}(y,\cdot))\leq\frac34.
		\]
	\end{proposition}
	
	
	\subsection{Isoperimetric inequality}\label{ssec:isoperimetry}
	
	A strongly log-concave distribution satisfies the Gaussian enlargement
	inequality below; see \citet[Corollary~2.2 and equations~(2.10)--(2.12)]{BakryLedoux1996}.
	It remains valid under the first-order assumptions
	\eqref{eq:first-order-assumptions}. To see this without assuming second
	differentiability, note that \(U(x)-m\norm{x}^2/2\) is convex.
	Caffarelli's contraction theorem \citep{Caffarelli2000}, in the nonsmooth
	formulation of \citet[Theorem~1.1, Lemma 3.3, and the conclusion of Section~3]{KimMilman2012},
	therefore gives a \(1\)-Lipschitz map \(T\) sending
	\(N_d(0,m^{-1}I_d)\) to \(\pi\).
	For each Borel set \(A\) and \(r>0\), the inclusion
	\((T^{-1}(A))^r\subseteq T^{-1}(A^r)\) transfers the Gaussian
	isoperimetric inequality to \(\pi\). Consequently,
	\begin{equation}\label{eq:bakry-ledoux-enlargement}
		\pi(A^r)
		\geq
		\Phi\bigl(\Phi^{-1}(\pi(A))+\sqrt{m}\,r\bigr),
	\end{equation}
	where, for \(r>0\),
	\(A^r=\{x:\operatorname{dist}(x,A)<r\}\), and \(\Phi\) is the standard
	Gaussian distribution function. The cases \(\pi(A)=0,1\) are understood by
	continuity. The same bound holds for sets measurable in the completion of
	\(\pi\), by choosing a Borel subset of the same measure.
	The following three-set form of the isoperimetric inequality is proved
	in \cref{app:gaussian}.
	
	\begin{proposition}\label{prop:separated}
		Let \(A,B,C\) be a measurable partition of \(\R^d\), let \(r\geq0\), and suppose
		\(\operatorname{dist}(A,B)\geq r\).  If
		\(0<q=\min\{\pi(A),\pi(B)\}\leq1/2\), then
		\begin{equation}\label{eq:separated}
			\pi(C)
			\geq
			\frac{q}{4}
			\min\left\{1,r\sqrt{m\log(1/q)}\right\}.
		\end{equation}
	\end{proposition}
	
	
	\subsection{One-step flow}
	
	For a $\pi$-reversible Markov kernel \(K\), define the one-step flow
	\[
	\mathcal{J}_K(S,S^c)=\int_S K(x,S^c)\pi(\df x).
	\]

	Combining \cref{prop:overlap} and \cref{prop:separated}, we establish the following result concerning $\cK_t$ in
	\cref{app:defective}.
	
	\begin{proposition} \label{prop:flow}
		One can find universal constants $b_0 \in (0,1/2]$ and $A_0 \in [2,\infty)$ such that the following holds.
		Let \(p\geq \widetilde p_\star:=A_0p_\star\), where $p_\star$ is defined in \cref{thm:main}.
		Let
		\(\theta\in(0,1]\),  and
		\[
		t=\frac{\theta b_0}{L\sqrt{p(d+p)}}.
		\]
		Then, for every measurable set $S\subseteq\R^d$ such that
		$e^{-p/2}\leq\pi(S)\leq1/2$,
		\[
		\mathcal{J}_{\cK_t}(S,S^c) \geq 2^{-12}\pi(S)
		\sqrt{mt\log[1/\pi(S)]},
		\]
		where $mt\log[1/\pi(S)]\leq1$.
		
		Moreover, for $t'\in(0,1/(2Ld)]$ and every measurable set
		$S\subseteq\R^d$ such that $0<\pi(S)\leq1/2$,
		\[
		\mathcal{J}_{\cK_{t'}}(S,S^c) \geq 2^{-12}\pi(S)
		\min\left\{1,\sqrt{mt'\log[1/\pi(S)]}\right\}.
		\]
	\end{proposition}
	
	Using \cref{prop:flow}, for a set $S \subset \mathbb{R}^d$ satisfying $0 < \pi(S) < 1/2$, we may bound the average flow $\mathcal{J}_{\cK_t}(S,S^c)/\pi(S)$ for appropriate values of~$t$.
	The smaller $\pi(S)$ is, the smaller~$t$ for which the average flow can be estimated via \cref{prop:flow}.

	\subsection{Aggregation of one-step flows} \label{ssec:aggregation}

	In \cref{sec:main-proof}, we choose disjoint intervals of the form
	\([t_j/2,t_j]\) contained in $[0,H]$, where \((t_0,\ldots,t_J)\) is a decreasing sequence of
	positive numbers.  For a given measurable set~\(S\) such that $0 < \pi(S) \leq 1/2$, the average flow
	\(\mathcal{J}_{\cK_{t_j}}(S,S^c) / \pi(S) \) may be large for some values of~\(j\) but small
	for others.  We therefore assign to each set~\(S\) an appropriate index
	\(j(S)\).
	Then, we establish a lower bound $\phi_{j(S)}$ on
	\(\mathcal{J}_{\cK_{t_{j(S)}}}(S,S^c) / \pi(S)\) using \cref{prop:flow}.  The purpose of this subsection is to provide the tools for
	combining these set-dependent bounds into a lower bound on
	\(\Gap(\overline P_H)\).
	
	Cheeger's inequality \citep{LawlerSokal1988} implies that the spectral gap of a reversible Markov kernel can be estimated by considering the kernel's average flow out of measurable sets.
	It may be tempting to use \cref{lem:Kt} with $f = \mathbf{1}_S$, where $0 < \pi(S) \leq 1/2$, to obtain
	\[
	\frac{\mathcal{J}_{\overline P_H}(S, S^c)}{\pi(S)} \geq \frac{1}{\pi(S)} \sum_{j = 0}^J \frac{t_j}{2H} \mathcal{J}_{\cK_{t_j}}(S, S^c) \geq \frac{t_{j(S)}}{2H} \frac{\mathcal{J}_{\cK_{t_{j(S)}}}(S, S^c)}{\pi(S)} \geq \frac{\min_j t_j \phi_j}{2H}.
	\]
	This gives a lower bound on the conductance of $\overline P_H$.
	It turns out that this direct estimate followed by the usual Cheeger
	inequality loses too much and does not yield the spectral lower bound of order $(\kappa \sqrt{d})^{-1}$ that we desire.
	We carry out the more careful approach below.

	The following lemma carries out this aggregation for a generic mixture of reversible Markov kernels.   
	It is possible to derive the lemma from the general
	Cheeger inequality of \citet{ChenWang2000}.  We give a direct proof in
	\cref{app:fractional}.
	
	\begin{lemma}
		\label{lem:fractional}
		Let \(P,K_1,\ldots,K_N\) be Markov kernels that are reversible with
		respect to a distribution \(\varpi\).  Suppose that there exist constants
		\(\gamma_1,\ldots,\gamma_N>0\) such that
		\[
		\cE_P(f,f)
		\geq
		\sum_{j=1}^N
		\gamma_j\cE_{K_j}(f,f)
		\qquad
		\text{for every }f\in L^2(\varpi).
		\]
		Suppose further that there exist constants \(\beta_1,\ldots,\beta_N\geq0\)
		satisfying
		\[
		\sum_{j=1}^N\frac{\beta_j^2}{\gamma_j}>0
		\]
		such that, for every measurable set \(S\) satisfying
		\(0<\varpi(S)\leq1/2\),
		\[
		\sum_{j=1}^N
		\beta_j \mathcal{J}_{K_j}(S,S^c)
		\geq
		\varpi(S),
		\]
		where
		\[
		\mathcal{J}_{K_j}(S,S^c)
		=
		\int_S K_j(x,S^c)\,\varpi(\df x).
		\]
		Then
		\[
		\Gap(P)
		\geq
		\left(
		2\sum_{j=1}^N\frac{\beta_j^2}{\gamma_j}
		\right)^{-1}.
		\]
	\end{lemma}
	
	The above lemma gives the following convenient hard-assignment
	formulation, which we use to combine the flow bounds in
	\cref{prop:flow}.
	Roughly speaking, \cref{thm:aggregation}
	allows the kernel that supplies the useful one-step flow bound to depend on
	the set under consideration, while keeping track of how often each kernel
	is selected in the mixture. 
	
	\begin{theorem}
		\label{thm:aggregation}
		Let \(P,K_1,\ldots,K_N\) be Markov kernels that are reversible with
		respect to a distribution \(\varpi\).  Suppose
		\[
		\cE_P(f,f)
		\geq
		\sum_{j=1}^N
		\gamma_j\cE_{K_j}(f,f),
		\qquad
		\gamma_j>0,
		\]
		for every \(f\in L^2(\varpi)\).
		
		Suppose further that there exist constants
		\(\phi_1,\ldots,\phi_N>0\) such that every measurable set \(S\) with
		\(0<\varpi(S)\leq1/2\) can be assigned an index
		\(j(S)\in\{1,\ldots,N\}\) satisfying
		\[
		\mathcal{J}_{K_{j(S)}}(S,S^c)
		\geq
		\phi_{j(S)}\varpi(S).
		\]
		Then
		\[
		\Gap(P)
		\geq
		\left(
		2\sum_{j=1}^N
		\frac{1}{\gamma_j\phi_j^2}
		\right)^{-1}.
		\]
	\end{theorem}
	
	\begin{proof}
		Apply \cref{lem:fractional} with
		\[
		\beta_j=\frac1{\phi_j},
		\qquad j=1,\ldots,N.
		\]
		For every measurable \(S\) with \(0<\varpi(S)\leq1/2\),
		\[
		\begin{aligned}
			\sum_{j=1}^N
			\beta_j\mathcal{J}_{K_j}(S,S^c)
			&\geq
			\frac{
				\mathcal{J}_{K_{j(S)}}(S,S^c)
			}{
				\phi_{j(S)}
			}\\
			&\geq
			\varpi(S).
		\end{aligned}
		\]
		The conclusion follows directly from \cref{lem:fractional}.
	\end{proof}

	\section{Proof of the Main Theorem}\label{sec:main-proof}
	
	\begin{proof}[Proof of \cref{thm:main}]
		Let $b_0\leq1/2$ and $A_0\geq 2$ be given in \cref{prop:flow}, and
		put $\widetilde p_\star=A_0p_\star$, where $p_\star$ is defined in \cref{thm:main}.  We divide the proof into
		three parts.
		
		\bigskip
		
		\noindent{\it Part 1: The geometric moment ladder.}
		
		Suppose for now that \(\widetilde p_\star<d\).  Let
		\begin{equation}\label{eq:pj}
			p_j=4^j \widetilde p_\star,
			\qquad j=0,1,\ldots,J,
		\end{equation}
		where \(J\) is the smallest integer for which \(p_J\geq d\).  Thus
		\begin{equation}\label{eq:pJ-range}
			d\leq p_J<4d.
		\end{equation}
		Define the nominal scales
		\begin{equation*}
			\tau_j=
			\frac{b_0}{L\sqrt{p_j(d+p_j)}},
		\end{equation*}
		and put
		\begin{equation}\label{eq:theta-tj}
			\theta=\min\left\{1,\frac{H}{\tau_0}\right\},
			\qquad
			t_j=\theta\tau_j.
		\end{equation}
		Then \(t_0\leq H\), and
		\[
		\frac{t_{j+1}}{t_j}
		=\frac12\sqrt{\frac{d+p_j}{d+4p_j}}
		<\frac12.
		\]
		Consequently, the Borel sets $[t_j/2,t_j]$
		are pairwise disjoint, have Lebesgue measure \(t_j/2\), and lie within $[0,H]$.
		Applying \cref{lem:Kt} to this finite
		family gives
		\begin{equation}\label{eq:ladder-dirichlet}
			\cE_{\overline P_H}(f,f)
			\geq
			\sum_{j=0}^J\frac{t_j}{2H}\cE_{\cK_{t_j}}(f,f).
		\end{equation}
		This prompts us to apply \cref{thm:aggregation}.
		
		Let \(S\) be measurable with \(0<\pi(S)\leq1/2\).
		We assign $S$ to a component $\cK_{t_j}$.
		That is, we identify the index $j(S)$ in \cref{thm:aggregation}.
		Since \(A_0\geq2\) and \(p_\star\geq1\),
		\[
		\frac{p_0}{2}=\frac{\widetilde p_\star}{2}
		\geq1>\log2.
		\] 
		The intervals
		\([\log 2,p_0/2]\), \((p_{j-1}/2,p_j/2]\) for \(1\leq j\leq J\), and
		\((p_J/2,\infty)\) are disjoint and cover \([\log 2,\infty)\).  Thus the
		following assignment is exhaustive and unambiguous.
		
		\begin{itemize}[leftmargin=2em]
			\item If \(\log 2\leq\log[1/\pi(S)]\leq p_0/2\), assign \(S\) to \(\cK_{t_0}\).
			By the first assertion of \cref{prop:flow},
			\[
			\mathcal{J}_{\cK_{t_0}}(S,S^c)
			\geq
			\phi_0\pi(S),
			\qquad
			\phi_0=2^{-12}\sqrt{m t_0\log 2}.
			\]
			
			\item For \(1\leq j\leq J\), if
			\(p_{j-1}/2<\log[1/\pi(S)]\leq p_j/2\), assign \(S\) to \(\cK_{t_j}\).
			By the first assertion of \cref{prop:flow},
			\[
			\mathcal{J}_{\cK_{t_j}}(S,S^c)
			\geq
			\phi_j\pi(S),
			\qquad
			\phi_j=2^{-12} \sqrt{\frac{m t_jp_{j-1}}{2}}= 2^{-12} \sqrt{\frac{m t_jp_j}{8}}.
			\]
			
			\item If \(\log[1/\pi(S)]>p_J/2\), assign \(S\) to the last component
			\(\cK_{t_J}\).  By \eqref{eq:pJ-range} and the assumption $b_0\leq1/2$,
			\[
			t_J
			\leq
			\frac{b_0}{L\sqrt{p_J(d+p_J)}}
			\leq
			\frac{1}{2Ld}.
			\]
			By the second assertion of \cref{prop:flow}, 
			\[
			\mathcal{J}_{\cK_{t_J}}(S,S^c)
			\geq
			\phi_J' \pi(S),
			\qquad
			\phi_J'= 2^{-12}  \min \left\{1, \sqrt{\frac{m t_Jp_J}{2}} \right\}.
			\]
			Moreover,
			\[
			\frac{mt_Jp_J}{2}
			\leq\frac{b_0}{2\kappa}\sqrt{\frac{p_J}{d+p_J}}
			\leq\frac{b_0}{2}<1.
			\]
			Weakening the flow bound by a constant factor, we obtain
			\[
			\mathcal{J}_{\cK_{t_J}}(S,S^c)
			\geq
			\phi_J \pi(S),
			\qquad
			\phi_J= 2^{-12}  \sqrt{\frac{m t_Jp_J}{8}}.
			\]
		\end{itemize}
		
		We may therefore apply \cref{thm:aggregation} to 
		\eqref{eq:ladder-dirichlet}, with
		\[
		\gamma_j=\frac{t_j}{2H},
		\qquad
		\phi_0^2= 2^{-24} m t_0\log 2,
		\qquad
		\phi_j^2= 2^{-27} m t_jp_j\quad(j\geq1).
		\]
		The summation estimate in \cref{lem:ladder-sum} yields
		\begin{equation*}
			\sum_{j=0}^J\frac1{\gamma_j\phi_j^2}
			\leq
			C\frac{HL^2 \widetilde p_\star (d+ \widetilde p_\star )}
			{m\theta^2b_0^2},
		\end{equation*}
		where \(C \in (0,\infty)\) is universal.  Hence,
		\begin{equation}\label{eq:ladder-gap-bound}
			\Gap(\overline P_H)
			\geq
			\frac{m\theta^2b_0^2}
			{2CHL^2 \widetilde p_\star (d+ \widetilde p_\star )}
			=
			\frac{m}{2CH}
			\min\left\{
			H,
			\frac{b_0}{L\sqrt{ \widetilde p_\star (d+ \widetilde p_\star )}}
			\right\}^{2}.
		\end{equation}
		
		\bigskip
		
		\noindent{\it Part 2: A conservative bound}
		
		The bound \eqref{eq:ladder-gap-bound} is almost what we need, but it has been
		established only for $ \widetilde p_\star <d$.  To obtain the exact bound in
		\cref{thm:main}, we derive a lower bound on $\Gap(\overline P_H)$ using only
		conservative step sizes.
		
		Define the safe step size
		\[
		t_{\mathrm s}=\min\left\{H,\frac{b_0}{Ld}\right\}.
		\]
		Because \(b_0\leq1/2\), $t_{\mathrm s} \leq 1/(2Ld)$.
		For every measurable \(S\) with \(0<\pi(S)\leq1/2\), apply the second assertion of
		\cref{prop:flow} to obtain
		\begin{equation} \label{eq:flow-safe}
			\mathcal{J}_{\cK_{t_{\mathrm s}}}(S,S^c) \geq  2^{-12}  \pi(S) \, \min \left\{ 1, \sqrt{mt_{\mathrm s} \log[1/\pi(S)]}  \right\} \geq  2^{-12}  \pi(S) \, \min \left\{ 1, \sqrt{mt_{\mathrm s} \log 2}  \right\}.
		\end{equation}
		Since
		\[
		m t_{\mathrm s}\log 2
		\leq\frac{b_0\log 2}{\kappa d}<1,
		\]
		the bound \eqref{eq:flow-safe} becomes
		\[
		\mathcal{J}_{\cK_{t_{\mathrm s}}}(S,S^c) \geq  2^{-12} \pi(S)\sqrt{mt_{\mathrm s}\log 2}.
		\]
		By \cref{lem:Kt},
		\[
		\cE_{\overline P_H}(f,f)
		\geq
		\frac{t_{\mathrm s}}{2H}\cE_{\cK_{t_{\mathrm s}}}(f,f).
		\]
		Apply \cref{thm:aggregation} with $N=1$, $K_1=\cK_{t_{\mathrm s}}$,
		$\gamma_1=t_{\mathrm s}/(2H)$, and
		$\phi_1= 2^{-12} \sqrt{mt_{\mathrm s}\log 2}$ to obtain
		\begin{equation}\label{eq:safe-gap-bound}
			\Gap(\overline P_H)
			\geq
			 2^{-26} (\log 2)\frac{m t_{\mathrm s}^{2}}{H}
			= 2^{-26} (\log 2)\frac{m}{H}
			\min\left\{H,\frac{b_0}{Ld}\right\}^2.
		\end{equation}
		
		\bigskip

		\noindent{\it Part 3: Completion of the proof}

		If $\widetilde p_\star\geq d$, then
		\[
		\frac1{\sqrt{\widetilde p_\star(d+\widetilde p_\star)}}\leq\frac1d.
		\]
		If $\widetilde p_\star<d$,
		combine \eqref{eq:ladder-gap-bound} and \eqref{eq:safe-gap-bound}, using
		\[
		\max\{\min(H,a)^2,\min(H,b)^2\}
		=\min\{H,\max(a,b)\}^{2}
		\qquad(a,b,H>0).
		\]
		In either case, with the universal constant
		$c_1=\min\{(2C)^{-1}, 2^{-26} \log2\}>0$, we obtain
		\[
		\Gap(\overline P_H)
		\geq c_1\frac{m}{H}
		\min\left\{H,\frac{b_0}{L}
		\max\left\{
		\frac1{\sqrt{\widetilde p_\star(d+\widetilde p_\star)}},\frac1d
		\right\}\right\}^{2}.
		\]
		Since $A_0\geq1$ and $\widetilde p_\star=A_0p_\star$,
		\[
		\widetilde p_\star(d+\widetilde p_\star)
		=A_0p_\star(d+A_0p_\star)
		\leq A_0^2p_\star(d+p_\star),
		\]
		and, because $b_0 \leq 1$, it follows that
		\[
		\Gap(\overline P_H)
		\geq c_1\left(\frac{b_0}{A_0}\right)^2\frac{m}{H}
		\min\left\{H,\frac1L\max\left\{
		\frac1{\sqrt{p_\star(d+p_\star)}},\frac1d
		\right\}\right\}^{2}.
		\]
		Choosing
		$c_0=(c_1/2)(b_0/A_0)^2$
		proves \eqref{eq:master-gap}.
		
		The identity kernel has zero Dirichlet form, so
		\[
		\Gap\bigl((I+\overline P_H)/2\bigr)
		=\frac12\Gap(\overline P_H).
		\]
		
	\end{proof}

	\section{Lean verification}\label{sec:lean-verification}
	
	The accompanying Lean~4 development \citep{deMouraUllrich2021},
	available in the GitHub repository \citep{qinMALALean2026}, formalizes
	all stated mathematical results of this article except
	\cref{lem:linear-increment,lem:integrated-increments,lem:frozen-endpoint-law,lem:path-likelihood}.
	The main first-order input consists of a continuously differentiable
	potential, the strong-convexity inequality in
	\eqref{eq:first-order-assumptions}, and a Lipschitz bound on its
	gradient.

	The development connects these inputs to the spectral-gap and mixing
	bounds in \cref{thm:main,cor:sqrt-d-endpoint,cor:mixing}, as well as the central limit theorems, asymptotic-variance bounds and comparison,
	and nonstationary mean-square error bounds in \cref{cor:asymptotic-variance,cor:variance-separation,cor:nonstationary-variance}. It also includes the smooth
	hard-target construction and fixed-step minimax obstruction in
	\cref{prop:generic-fixed-step-obstruction,prop:minimax-fixed-step-ceiling},
	together with the fixed-step spectral-gap lower bound in
	\cref{prop:small-fixed-step-gap}, summarized in
	\cref{rem:minimax-fixed-step-ceiling}. The fractional aggregation
	lemma and component-aggregation theorem,
	\cref{lem:fractional,thm:aggregation}, are formalized as reusable results
	independently of their application to MALA. The formalized algorithms and
	quantities include the target distribution, proposal and acceptance
	rules, randomized and lazy transition kernels, spectral gaps, trajectory
	laws, sample means, and asymptotic variances.
	
	The stationary-rejection and overlap conclusions in
	\cref{prop:stationary-rejection,prop:overlap} use an alternative
	discrete-time proof based on finite-dimensional Gaussian likelihood
	estimates, an Euler/random-walk-Metropolis comparison, and weak-limit
	stability. This route bypasses the continuous-time results in \cref{lem:linear-increment,lem:integrated-increments,lem:frozen-endpoint-law,lem:path-likelihood}.
	\cref{prop:stationary-rejection} is formalized under the
	nonconvex assumptions stated at the beginning of \cref{app:rejection-overlap}.
	
	A reusable Bakry--Ledoux enlargement bound for targets satisfying
	the standing first-order assumptions is derived internally from a
	Gaussian Ornstein--Uhlenbeck argument and a finite-Euler weak limit.
	

	The sole direct external Lean library is mathlib \citep{Mathlib2020};
	its standard dependencies and Lean's logical foundations are retained.
	The complete development was kernel-checked with Lean~4.33.0 and
	mathlib~4.33.0. The full build and the dependency audit of 318 selected
	declarations succeeded, with no proof placeholders or project-specific
	axioms. The audited declarations depend only on the standard axioms
	\texttt{propext}, \texttt{Classical.choice}, and \texttt{Quot.sound}.
	
	The package documentation provides a statement-by-statement theorem
	map, a guide to comparing the definitions with the paper, and routes
	through the complete proof dependencies. 
	
	\section{Numerical experiments}\label{sec:simulations}
	
	We compare fixed-step MALA and uniformly randomized-step MALA on a
	perturbed-Gaussian target.
	The simulation package can be found in a GitHub repository \citep{qinMALALean2026}.
	
	Let $h_0$ be a positive number and consider a potential function
	\begin{equation}\label{eq:simulation-hard-target}
		U_{d,h_0}(x)
		=\frac{m}{2}x_1^2
		+\sum_{i=2}^d\left[
		\frac{L+m}{4}x_i^2
		-\frac{(L-m)h_0}{2}
		\cos\left(\frac{x_i}{\sqrt{h_0}}\right)
		\right].
	\end{equation}
	In \cref{app:fixed-step-obstruction}, perturbed Gaussian distributions with densities proportional to $e^{-U_{d,h}}$ are investigated from a theoretical perspective.
	As shown in \cref{prop:generic-fixed-step-obstruction},
	\(mI_d\preceq\nabla^2U_{d,h_0}(x)\preceq LI_d\) for every \(x\).
	Throughout, \(m=0.1\) and \(L=1\), so \(\kappa=10\).
	This type of construction appears in \citep{ChewiEtAl2021,LeeShenTian2021,WuSchmidlerChen2022}, and it serves as a hard target for fixed-step MALA.
	
	For an endpoint \(H>0\), we compare a fixed-step MALA $P_h$ with \(h=H\) and the randomized-step $\overline P_H$.  We report the mean
	acceptance probability $\E\left[
	\alpha_h(X,Y)
	\right]$
	and the expected squared jumping distance per
	coordinate,
	\begin{equation}\nonumber
		\operatorname{ESJD}
		=\frac1d\E\left[
		\alpha_h(X,Y)\norm{Y-X}^2
		\right],
	\end{equation}
	where $X$ is the current state, and $Y$ is the proposed state.
	For fixed-step MALA, $Y \mid X,h \sim Q_h(X, \cdot)$ where $h = H$ is fixed.
	For randomized-step MALA, $Y \mid X, h \sim Q_h(X, \cdot)$ and $h \sim \mathrm{Unif}(0,H)$.
	We will investigate the setting $X = 0$ as well as $X \sim \pi$.
	To estimate the mean acceptance rate and the ESJD, we use classical Monte Carlo.
	That is, independent copies of $(X,Y,h)$ are generated, and expectations under the empirical distribution are used as estimates.
	
	\subsection{Acceptance at the origin}
	
	For each \(d\in\{25,50,100,200,400\}\), we set
	\[
	h_0=\frac{4}{L\sqrt d}
	\]
	and compare fixed-step MALA with \(h=h_0\) against randomized-step MALA with
	\(H=h_0\) in terms of the mean acceptance rate at the origin.
	As shown in \cref{app:fixed-step-obstruction}, the origin is where fixed-step MALA struggles to leave.
	Each estimate uses \(80{,}000\)
	independent proposals.
	
	\begin{figure}[tbp]
		\centering
		\includegraphics[width=0.72\textwidth]
		{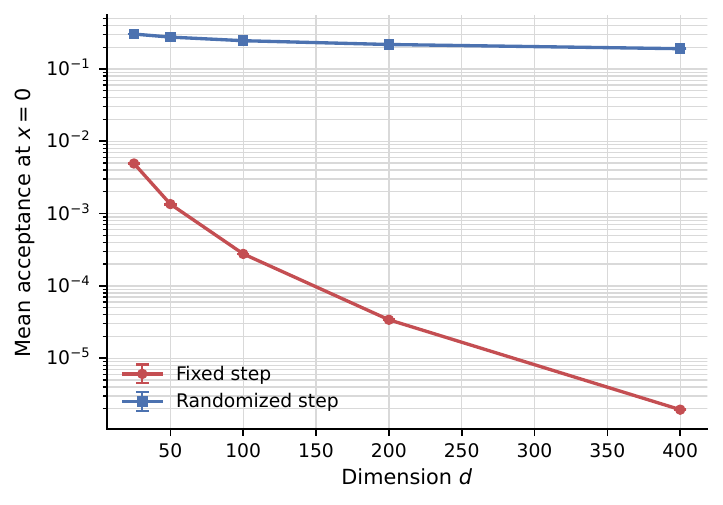}
		\caption{Estimated mean acceptance probability at the origin for the target
			\eqref{eq:simulation-hard-target}.  Fixed-step MALA uses \(h=h_0\), while
			randomized-step MALA uses \(h\sim\operatorname{Unif}(0,h_0)\).  The vertical
			axis is logarithmic, and vertical bars show \(\pm2\) estimated standard
			errors; most are narrower than the markers.}
		\label{fig:hard-target-origin-acceptance}
	\end{figure}
	
	As \cref{fig:hard-target-origin-acceptance} shows, the fixed-step acceptance
	probability decreases from approximately \(4.92\times10^{-3}\) at \(d=25\)
	to \(1.94\times10^{-6}\) at \(d=400\).  The randomized-step acceptance
	probability decreases much more slowly, from \(0.304\) to \(0.190\).
	Because rejection leaves a chain initialized at the origin unchanged, the
	time to its first accepted move is geometric.  At \(d=400\), the reciprocal
	acceptance estimates imply mean waiting times of approximately
	\(5.14\times10^5\) iterations for fixed-step MALA and \(5.25\) iterations
	for randomized-step MALA.  The experiment therefore directly displays the
	sticky behavior behind the fixed-step obstruction: proposals at the
	oscillation scale \(h_0\) have rapidly deteriorating acceptance, whereas
	uniform randomization retains a substantial probability of selecting a
	smaller step.

	To continue the origin experiment beyond one step, we next examine what
	happens after fixed-step MALA eventually escapes the hard point.  We run one
	chain of each type for \(40{,}000\) iterations on the \(d=200\) target,
	starting both chains at the origin and taking \(H=h_0\).  
	We record the first
	coordinate at every iteration.
	
	\begin{figure}[htbp]
		\centering
		\includegraphics[width=\textwidth]
		{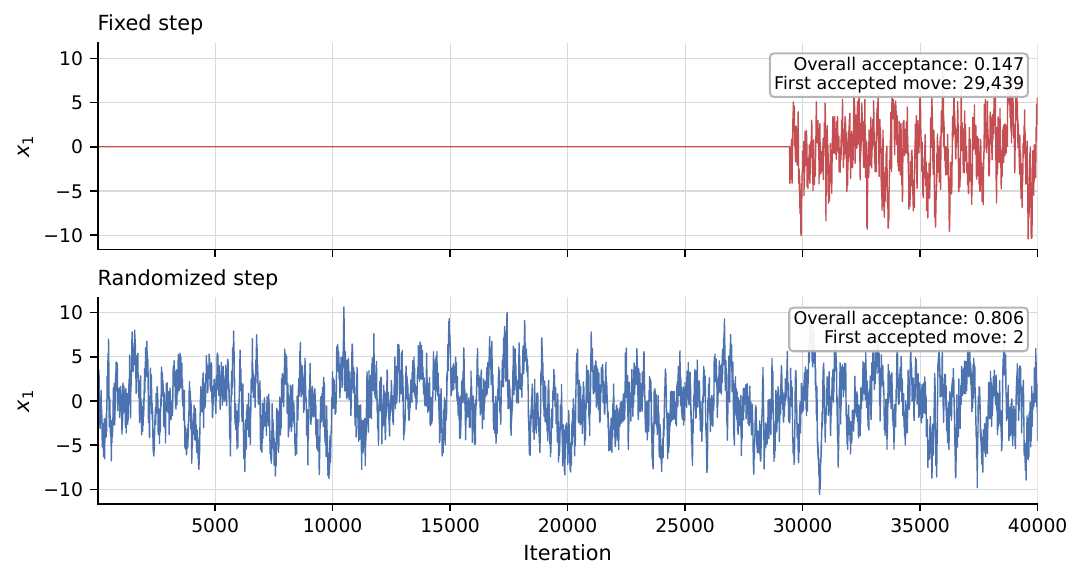}
		\caption{First-coordinate traces continuing the origin experiment for the
			\(d=200\) hard target, with both chains initialized at the origin and
			\(H=h_0=4/\sqrt{200}\).  The two panels use common horizontal and vertical
			scales.}
		\label{fig:hard-target-first-coordinate-trace}
	\end{figure}
	
	In this realization, the fixed-step chain rejects its first \(29{,}438\)
	proposals and first moves at iteration \(29{,}439\).  It then makes
	\(5{,}882\) accepted moves by iteration \(40{,}000\), so its overall
	acceptance rate is \(0.147\).  Over the \(10{,}562\)-iteration interval from
	its first accepted move through the end of the run, its empirical acceptance
	rate is \(0.557\).  The trace therefore shows that, once the fixed chain
	escapes the sticky point, it begins moving on an ordinary scale; the low
	overall acceptance rate is dominated by the long initial wait.  The
	randomized chain first moves at iteration \(2\), makes \(32{,}253\) accepted
	moves, and has overall acceptance rate \(0.806\).

	\subsection{Stationary acceptance and jumping distance}
	
	The preceding experiment deliberately starts at an atypical point.  We next
	compare average one-step behavior under stationarity.  We fix \(d=200\) and
	\(h_0=4/\sqrt{200}\), and vary \(H/h_0\) from \(0.05\) to \(100\), using a
	fine grid near the efficient range and a logarithmically spaced grid on the
	over-tuning side.
	
	Exact stationary samples are available because
	\eqref{eq:simulation-hard-target} is a product target.  The first coordinate
	is \(N(0,m^{-1})\).  For each remaining coordinate, we propose
	\(\widetilde X\sim N(0,2/(L+m))\) and accept it with probability
	\[
	\exp\left\{
	\frac{(L-m)h_0}{2}
	\left[
	\cos\left(\frac{\widetilde X}{\sqrt{h_0}}\right)-1
	\right]
	\right\}.
	\]
	This is an exact rejection sampler.  At every endpoint~$H$ and for each method,
	we use \(48{,}000\) independent stationary one-step experiments.
	
	\begin{figure}[tbp]
		\centering
		\includegraphics[width=\textwidth]
		{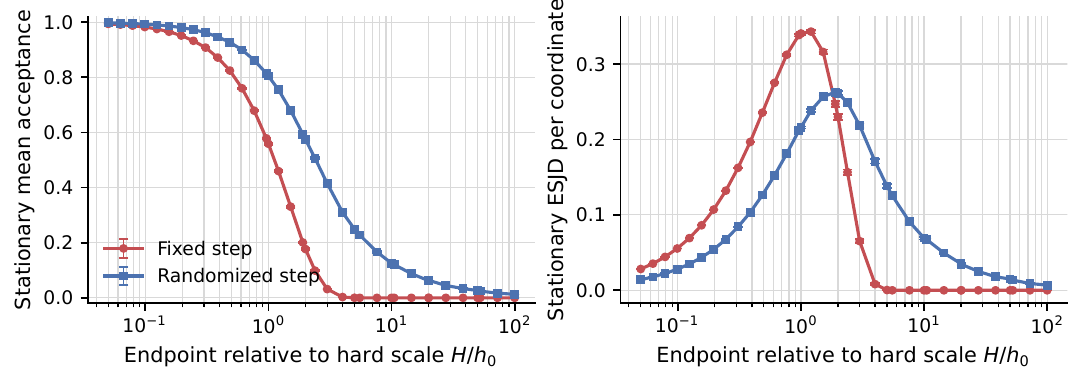}
		\caption{Stationary one-step performance for the \(d=200\) target with
			\(h_0=4/\sqrt{200}\).  Fixed-step MALA uses \(h=H\), while randomized-step
			MALA uses \(h\sim\operatorname{Unif}(0,H)\).  Vertical bars show
			\(\pm2\) estimated standard errors.}
		\label{fig:hard-target-stationary}
	\end{figure}
	
	\Cref{fig:hard-target-stationary} shows the expected
	efficiency--robustness tradeoff.  At \(H=h_0\), fixed-step MALA has
	acceptance \(0.559\) and ESJD \(0.340\), while randomized-step MALA has
	acceptance \(0.804\) and ESJD \(0.216\).  The largest observed fixed-step
	ESJD is \(0.343\), at \(H/h_0\approx1.21\), whereas the largest observed
	randomized-step ESJD is \(0.263\), at \(H/h_0\approx1.90\).  Thus a
	well-tuned fixed step has the larger peak ESJD.
	
	Randomization becomes advantageous when $H$ is too large.  
	At \(H=10h_0\), fixed-step acceptance is
	extremely close to~0, whereas randomized-step MALA retains
	acceptance \(0.125\) and ESJD \(0.0692\).  Even at \(H=100 \, h_0\), the
	randomized method has acceptance \(0.0127\) and ESJD \(0.00681\), while the
	corresponding fixed-step estimates are below floating-point resolution.
	Randomization does not prevent performance from eventually decreasing as
	\(H\) grows, but it changes the abrupt fixed-step collapse into a much more
	gradual decline.
	
	The two experiments also expose a distinction relevant to the global
	spectral gap.  When \(d=200\) and \(H=h_0\), fixed-step acceptance is about
	\(0.559\) under stationarity but only \(3.40\times10^{-5}\) at the origin.
	Thus stationary-average diagnostics largely miss the small sticky region, which leads to the exceedingly small spectral gap of fixed-step MALA, as shown in \cref{app:fixed-step-obstruction}.
	Uniform randomization improves
	acceptance both in that region and under over-tuning, although it does not
	dominate a well-tuned fixed step in stationary ESJD.  These one-step
	experiments illustrate tuning robustness; they do not estimate the spectral
	gap itself.

	\bigskip
	
	\noindent{\bf Acknowledgment:} The author thanks Guanyang Wang for reading the paper and for providing helpful comments.
	
	\vspace{2cm}
	
	\noindent{\LARGE \bf Appendices}
	
	\appendix
	
	\section{Proof of \cref{prop:minimax-fixed-step-ceiling}}
	\label{app:fixed-step-obstruction}

	We follow an existing line of research on hard targets for fixed-step
	MALA. Perturbed-Gaussian constructions appear in
	\citet[Theorems~8, 9, and~36]{ChewiEtAl2021}.
	The product structure combining a Gaussian coordinate with
	cosine-perturbed Gaussian coordinates appears in
	\citet[Supplement, equation~(20)]{LeeShenTian2021} and
	\citet[Equation~(28) and Lemma~8]{WuSchmidlerChen2022}.
	Below, we use a modified version of the latter construction to obtain
	the explicit bound in \cref{prop:minimax-fixed-step-ceiling}.

	\begin{proposition}
		\label{prop:generic-fixed-step-obstruction}
		There are universal constants $c,C>0$ such that the following holds.
		Let $d\geq2$, $0<m<L$, and $h>0$.  Define
		\begin{equation}
			\label{eq:generic-hard-potential}
			U_{d,h}(x)
			=\frac{m}{2}x_1^2
			+\sum_{i=2}^d\left[
			\frac{L+m}{4}x_i^2
			-\frac{(L-m)h}{2}
			\cos\left(\frac{x_i}{\sqrt h}\right)
			\right].
		\end{equation}
		Then $U_{d,h}$ is infinitely differentiable and
		\begin{equation}
			\label{eq:generic-hard-curvature}
			mI_d\preceq\nabla^2U_{d,h}(x)\preceq LI_d,
			\qquad x\in\R^d.
		\end{equation}
		If $P_h$ is the MALA kernel with step size~$h$ and target density
		proportional to $e^{-U_{d,h}}$, then
		\begin{equation}
			\label{eq:generic-fixed-step-gap-upper}
			\Gap(P_h)
			\leq
			C\min\left\{
			mh+\frac{(mh)^2}{2},
			\exp\left[
			-c(d-1)\min\{(L-m)h,1\}
			\right]
			\right\}.
		\end{equation}
		Writing $\kappa=L/m$, this further implies, after possibly changing
		the universal constants $c$ and~$C$, that
		\begin{equation}
			\label{eq:generic-fixed-step-gap-upper-kappa}
			\Gap(P_h)
			\leq
			C\min\left\{
			\kappa^{-1}Lh,
			\exp\left[
			-cd\min\{(1-\kappa^{-1})Lh,1\}
			\right]
			\right\}.
		\end{equation}
	\end{proposition}
	
	\begin{proof}
		For $i\geq2$, the $i$th diagonal entry of the Hessian of
		$U_{d,h}$ is
		\[
		\frac{L+m}{2}
		+\frac{L-m}{2}
		\cos\left(\frac{x_i}{\sqrt h}\right),
		\]
		which lies in $[m,L]$.  The first diagonal entry equals~$m$, and
		all off-diagonal entries vanish.  This proves
		\eqref{eq:generic-hard-curvature}.
		
		We first establish the local upper bound.  The target is a product
		measure, and its first marginal is $N(0,m^{-1})$.  Use
		$f(x)=x_1$ in the Poincar\'e quotient.  Let $X$ have the target
		distribution, let $Z\sim N_d(0,I_d)$ be independent of~$X$, and
		write
		\[
		\widetilde Y
		=X-h\nabla U_{d,h}(X)+\sqrt{2h}\,Z
		\]
		for the MALA proposal.  The actual transition either moves to
		$\widetilde Y$ or remains at~$X$.  Consequently,
		\begin{align*}
			\cE_{P_h}(f,f)
			&\leq
			\frac12\E\left[
			\bigl(-hmX_1+\sqrt{2h}\,Z_1\bigr)^2
			\right]\\
			&=\frac12\left(h^2m+2h\right),
		\end{align*}
		where we used $\E(X_1^2)=m^{-1}$ and the independence of
		$X_1$ and~$Z_1$.  Since $\Var_\pi(f)=m^{-1}$, it follows that
		\begin{equation}
			\label{eq:generic-local-upper}
			\Gap(P_h)
			\leq
			mh+\frac{(mh)^2}{2}.
		\end{equation}
		
		We next obtain the exponential upper bound by controlling the mean
		acceptance probability at the origin.  Write
		\[
		a=\frac{L+m}{2},
		\qquad
		b=\frac{L-m}{2}.
		\]
		At the origin, $\nabla U_{d,h}(0)=0$, so a proposal has the form
		\[
		Y=\sqrt{2h}\,Z.
		\]
		For a general differentiable potential~\(U\), the log Hastings ratio
		for a proposal from \(x\) to \(y\) has the following standard
		representation; see also
		\citet[Proposition~1.3, equation~(1.33)]{Eberle2014}.
		\begin{align}
			R_h(x,y)
			&:=\log
			\frac{e^{-U(y)}q_h(y,x)}
			{e^{-U(x)}q_h(x,y)}
			\notag\\
			&=U(x)-U(y)
			+\frac12
			\ip{y-x}{\nabla U(x)+\nabla U(y)}
			\notag\\
			&\quad
			+\frac h4\left[
			\norm{\nabla U(x)}^2
			-\norm{\nabla U(y)}^2
			\right].
			\label{eq:generic-log-Hastings-ratio}
		\end{align}
		For $i=2,\ldots,d$, put
		\[
		V_i=\frac{Y_i}{\sqrt h}=\sqrt2Z_i.
		\]
		Substituting \eqref{eq:generic-hard-potential} into
		\eqref{eq:generic-log-Hastings-ratio} gives
		\begin{align}
			R_h(0,Y)
			={}&-\frac{m^2h^2}{2}Z_1^2
			+bh\sum_{i=2}^d A(V_i)
			\notag\\
			&-\frac{h^2}{4}\sum_{i=2}^d
			\bigl[aV_i+b\sin(V_i)\bigr]^2,
			\label{eq:generic-hard-log-ratio}
		\end{align}
		where
		\[
		A(v)=\cos v-1+\frac{v\sin v}{2}.
		\]
		Both the first and the last terms in
		\eqref{eq:generic-hard-log-ratio} are nonpositive.
		
		Since $V_i\sim N(0,2)$,
		\[
		\E(\cos V_i)=e^{-1},
		\qquad
		\E(V_i\sin V_i)=2e^{-1}.
		\]
		The first identity follows from the characteristic function of
		$V_i$, and the second follows from Stein's identity.  Hence
		\begin{equation*}
			\E[A(V_i)]
			=\frac2e-1
			=-a_0,
			\qquad
			a_0=1-\frac2e>0.
		\end{equation*}
		Moreover,
		\[
		\abs{A(v)}\leq2+\frac{\abs v}{2}.
		\]
		Thus $A(V_i)-\E[A(V_i)]$ has a universally bounded
		sub-Gaussian norm; see, for example,
		\citep[Section~2.5.2]{Vershynin2018}.  The concentration
		inequality for sums of independent sub-Gaussian random variables
		\citep[Theorem~2.6.3]{Vershynin2018} therefore yields a universal
		constant $c_0>0$ such that
		\begin{equation}
			\label{eq:generic-hard-concentration}
			\mathbb P\left(
			\sum_{i=2}^d A(V_i)
			>-\frac{a_0(d-1)}{2}
			\right)
			\leq e^{-c_0(d-1)}.
		\end{equation}
		Outside the event in \eqref{eq:generic-hard-concentration},
		\[
		R_h(0,Y)
		\leq
		-\frac{a_0}{2}bh(d-1)
		=
		-\frac{a_0}{4}(d-1)(L-m)h.
		\]
		Let
		\[
		a_h(x)=\int\alpha_h(x,y)Q_h(x,\df y)
		\]
		denote the pointwise acceptance probability.  It follows that
		\begin{align*}
			a_h(0)
			&=\E\left[1\wedge e^{R_h(0,Y)}\right]
			\\
			&\leq
			\exp\left[
			-\frac{a_0}{4}(d-1)(L-m)h
			\right]
			+e^{-c_0(d-1)}
			\\
			&\leq
			2\exp\left[
			-c_1(d-1)\min\{(L-m)h,1\}
			\right]
		\end{align*}
		for a universal constant $c_1>0$.
		
		The function $x\mapsto a_h(x)$ is continuous.  Indeed, writing
		the proposal as
		\[
		x-h\nabla U_{d,h}(x)+\sqrt{2h}\,Z,
		\qquad Z\sim N_d(0,I_d),
		\]
		the resulting acceptance probability is continuous in~$x$ for
		every fixed~$Z$ and is bounded by one.  Continuity therefore
		follows from dominated convergence.
		
		Consequently, there is an open ball~$S$ centered at the origin
		such that
		\[
		0<\pi(S)\leq\frac12,
		\qquad
		\sup_{x\in S}a_h(x)
		\leq
		4\exp\left[
		-c_1(d-1)\min\{(L-m)h,1\}
		\right].
		\]
		A transition starting in~$S$ can leave~$S$ only if its proposal
		is accepted.  Hence
		\[
		\mathcal{J}_{P_h}(S,S^c)
		\leq
		4\exp\left[
		-c_1(d-1)\min\{(L-m)h,1\}
		\right]\pi(S).
		\]
		Testing the Poincar\'e quotient with $\mathbf 1_S$ gives
		\begin{align}
			\Gap(P_h)
			&\leq
			\frac{\cE_{P_h}(\mathbf 1_S,\mathbf 1_S)}
			{\Var_\pi(\mathbf 1_S)}
			\notag\\
			&=
			\frac{\mathcal{J}_{P_h}(S,S^c)}
			{\pi(S)[1-\pi(S)]}
			\notag\\
			&\leq
			8\exp\left[
			-c_1(d-1)\min\{(L-m)h,1\}
			\right].
			\label{eq:generic-hard-sticky-gap}
		\end{align}
		Combining \eqref{eq:generic-local-upper} and
		\eqref{eq:generic-hard-sticky-gap} proves
		\eqref{eq:generic-fixed-step-gap-upper}.
		
		Finally, we prove
		\eqref{eq:generic-fixed-step-gap-upper-kappa}.  We have
		\[
		L-m=(1-\kappa^{-1})L,
		\qquad
		d-1\geq\frac d2.
		\]
		Therefore
		\[
		(d-1)\min\{(L-m)h,1\}
		\geq
		\frac d2
		\min\{(1-\kappa^{-1})Lh,1\}.
		\]
		If $mh\leq1$, then
		\[
		mh+\frac{(mh)^2}{2}
		\leq\frac32mh
		=\frac32\kappa^{-1}Lh.
		\]
		If $mh>1$, then
		\[
		\exp\left[
		-cd\min\{(1-\kappa^{-1})Lh,1\}
		\right]
		\leq1<mh=\kappa^{-1}Lh,
		\]
		so the exponential bound is the smaller of the two terms in
		\eqref{eq:generic-fixed-step-gap-upper-kappa}.  After adjusting
		the universal constants, this proves
		\eqref{eq:generic-fixed-step-gap-upper-kappa}.
	\end{proof}

	We can now prove \cref{prop:minimax-fixed-step-ceiling}.
	
	\begin{proof}
		By \cref{prop:generic-fixed-step-obstruction}, there are universal
		constants $c_0,C_0>0$ such that, for every $h>0$,
		\begin{equation}
			\label{eq:minimax-two-branches}
			\inf_{\substack{U\in C^\infty(\R^d)\\
					mI_d\preceq\nabla^2U(x)\preceq LI_d,\ x\in\R^d}}
			\Gap(P_{U,h})
			\leq F(Lh),
		\end{equation}
		where
		\[
		F(t)
		=
		C_0\min\left\{
		\kappa^{-1}t,\,
		\exp\left[
		-c_0d\min\{(1-\kappa^{-1})t,1\}
		\right]
		\right\}.
		\]
		Set
		\[
		t_\star
		=
		\frac{2\log(\kappa d)}
		{c_0(1-\kappa^{-1})d}.
		\]
		If $0<t\leq t_\star$, then
		\[
		F(t)
		\leq
		C_0\kappa^{-1}t_\star
		=
		\frac{2C_0\log(\kappa d)}
		{c_0(1-\kappa^{-1})\kappa d}.
		\]
		If $t>t_\star$, then the exponential term defining $F$ is
		nonincreasing in~$t$, and hence
		\begin{align*}
			F(t)
			&\leq
			C_0\exp\left[
			-c_0d\min\{(1-\kappa^{-1})t_\star,1\}
			\right]\\
			&=
			C_0\max\left\{
			\frac{1}{(\kappa d)^2},
			e^{-c_0d}
			\right\}.
		\end{align*}
		Consequently,
		\begin{equation}\label{eq:minimax-F-upper}
			\begin{aligned}
				\sup_{t>0}F(t)
				&\leq
				C_0\max\left\{
				\frac{2\log(\kappa d)}
				{c_0(1-\kappa^{-1})\kappa d},
				\frac{1}{(\kappa d)^2},
				e^{-c_0d}
				\right\} \\
				&\leq C_0\max\left\{
				\frac{2\log(\kappa d)}
				{c_0(1-\kappa_0^{-1})\kappa d},
				\frac{1}{(\kappa d)^2},
				e^{-c_0d}
				\right\}.
			\end{aligned}
		\end{equation}
		Because $d\geq2$ and $\kappa>1$,
		\[
		\frac{1}{(\kappa d)^2}
		\leq
		C\frac{\log(\kappa d)}{\kappa d}
		\]
		for a universal constant~$C$.  Absorbing the factors depending
		only on $\kappa_0$ into~$C$, and changing the universal
		exponential constant if necessary, \eqref{eq:minimax-two-branches} and \eqref{eq:minimax-F-upper}
		prove \eqref{eq:minimax-fixed-step-ceiling}.
	\end{proof}

	\section{Proof of \cref{prop:overlap}}\label{app:rejection-overlap}
	
	This appendix proves \cref{prop:overlap}.
	The convexity assumption on \(U\) is not needed for
	\cref{prop:stationary-rejection,lem:linear-increment,lem:integrated-increments,lem:frozen-endpoint-law,lem:path-likelihood} below.
	For these results it suffices that \(U\in C^1(\R^d)\), that
	\(\nabla U\) be globally \(L\)-Lipschitz for some \(L>0\), and that
	\(0<\int_{\R^d}e^{-U(x)}\,\df x<\infty\), with \(\pi\) the normalized
	measure and \(X_0\sim\pi\). 
	\cref{prop:overlap}, proved at the end of
	this appendix, continues to use convexity.
	
	The main input is the following stationary estimate for the pointwise rejection function
	\[
	r_h(x)=1-\int\alpha_h(x,y)Q_h(x,\df y),
	\]
	where
	\[
	\alpha_h(x,y)
	=1\wedge
	\frac{e^{-U(y)}q_h(y,x)}{e^{-U(x)}q_h(x,y)},
	\qquad
	q_h(x,y)=(4\pi h)^{-d/2}
	\exp\left\{-\frac{\norm{y-x+h\nabla U(x)}^2}{4h}\right\}.
	\]

	\begin{proposition} \label{prop:stationary-rejection}
		There are universal constants \(c,C>0\) such that, for every real
		\(p\geq 1\),
		\begin{equation}\label{eq:stationary-rejection}
			h\leq\frac{c}{L\sqrt{p(d+p)}}
			\quad\Longrightarrow\quad
			\norm{r_h}_{L^p(\pi)}
			\leq
			CLh\sqrt{p(d+p)}.
		\end{equation}
	\end{proposition}
	
	The proof of \cref{prop:stationary-rejection} compares the exact stationary Langevin path with the path whose
	drift is frozen at its initial position.  The endpoint of the latter path
	is exactly the MALA proposal.  Similar diffusion-comparison ideas appear in
	\citet{BouRabeeVandenEijnden2010,BouRabeeHairer2013,ChewiEtAl2021}.
	
	The idea behind the proof is roughly as follows.
	Let $X_0 \sim \pi$, and suppose that, on a filtered probability space $(\Omega,\mathcal F,(\mathcal F_t)_{t\geq0},\mathbb P)$, $X_t$ solves the Langevin equation
	\begin{equation} \label{eq:langevin}
		\df X_t = -\nabla U(X_t) \, \df t + \sqrt{2} \df B_t,
	\end{equation}
	where $B_t$ is a Brownian motion.
	Standard theory of Langevin diffusion implies $X_h \sim \pi$.
	Denote by $S_h$ the conditional distribution of $X_h$ given $X_0$.
	We compare $S_h(x, \df y)$ to the proposal $Q_h(x, \df y)$, which gives the Gaussian law $N_d(x - h \nabla U(x), 2h I_d)$.
	Define $\nu_h(\df x, \df y) = \pi(\df x) S_h(x, \df y)$, $\mu_h(\df x, \df y) = \pi(\df x) Q_h(x, \df y)$.
	Assume that the Radon-Nikodym derivative $F_h = \df \mu_h/\df \nu_h$ can be found.
	Since the Langevin diffusion is $\pi$-reversible, it holds that $\nu_h$ is symmetric, and
	\[
	\frac{\df \mu_h^{\top}}{\df \mu_h}(x,y) = \frac{F_h(y,x)}{F_h(x,y)},
	\]
	where $\mu_h^{\top}(\df x, \df y) = \mu_h(\df y, \df x)$.
	Moreover, it can be argued that, for $\pi$-a.e. $x$, $F_h(x, \cdot) = \df Q_h(x,\cdot)/\df S_h(x,\cdot)$.
	Then
	\[
	\begin{aligned}
		r_h(x) &= \int S_h(x, \df y) \left[ 1 - F_h(x,y) \min\left\{ 1, \frac{F_h(y,x)}{F_h(x,y)} \right\} \right]  \\
		&\leq \int(|F_h(x,y) - 1| + |F_h(y,x) - 1|) \, S_h(x, \df y).
	\end{aligned}
	\]
	Using Jensen's inequality and the fact that $\nu_h$ is symmetric, one can now show that $\|r_h\|_{L^p(\pi)}$ can be controlled by $2 \|F_h - 1\|_{L^p(\nu_h)}$.
	
	To find $F_h$, let 
	\[
	\df W_t = \df B_t - \theta_t \, \df t, \quad \theta_t = \frac{\nabla U(X_t) - \nabla U(X_0)}{\sqrt{2}}
	\]
	Then $X_t$ has the frozen drift representation $\df X_t = - \nabla U(X_0) \df t + \sqrt{2} \df W_t$.
	Suppose that one can find an alternative probability measure $\mathbb P_h^{\mathrm{fr}}$ defined by a Radon-Nikodym derivative $D_h = \df \mathbb P_h^{\mathrm{fr}}/\df \mathbb P|_{\mathcal{F}_h}$ such that, under $\mathbb P_h^{\mathrm{fr}}$, $W_t$ is a Brownian motion with respect to $(\mathcal{F}_t)_{0 \leq t \leq h}$.
	Then, under this new probability measure, $(X_0, X_h)$ is distributed as $\mu_h$.
	In particular, with expectations defined in terms of $\mathbb{P}$, $F_h(X_0, X_h) = \mathbb{E}[D_h \mid \sigma(X_0, X_h)]$, so conditional Jensen's inequality would imply that $\|F_h - 1\|_{L^p(\nu_h)} \leq \|D_h - 1\|_{L^p(\mathbb{P})}$.
	As stated in \cref{lem:frozen-endpoint-law}, one can construct the desired measure $\mathbb{P}_h^{\mathrm{fr}}$ by letting $D_h$ be defined through the stochastic exponential
	\[
	D_t = \exp\left( \int_0^t \langle \theta_s, \df B_s \rangle - \frac{1}{2} V_t \right), \quad V_t = \int_0^t \|\theta_s\|^2 \, \df s,
	\]
	provided that $(D_t)_{0 \leq t \leq h}$ is a true martingale.
	This is a consequence of  Girsanov's theorem \citep[see, e.g.,][]{KaratzasShreve1991,RevuzYor1999}.
	Therefore, to prove \cref{prop:stationary-rejection}, it comes down to two tasks: showing that $D_t$ is a true martingale, and that $\|D_h - 1\|_{L^p(\mathbb{P})} \leq C L h \sqrt{p(d+p)}$ when $h \leq c/(L\sqrt{p(d+p)})$, where $c, C$ are appropriate universal constants.
	
	To show that $D_t$ is a martingale, it suffices to establish Novikov's criterion: $\mathbb{E} (e^{V_h/2}) < \infty$.
	Since $\nabla U$ is $L$-Lipschitz, one can bound $V_h$ by $L^2 J_h/2$, where $J_h = \int_0^h \|X_t - X_0\|^2 \, \df t$.
	This prompts us to bound $\mathbb{E}(e^{\lambda J_h})$ for some suitable $\lambda \geq 0$ in \cref{lem:integrated-increments}.
	
	We obtain the upper bound $\|D_h - 1\|_{L^p(\mathbb{P})} \leq C L h \sqrt{p(d+p)}$ for $h \leq c/(L\sqrt{p(d+p)})$ in \cref{lem:path-likelihood}.
	In the proof of \cref{lem:path-likelihood}, we apply It\^o's formula, Burkholder-Davis-Gundy (BDG) inequality, and Doob's maximal inequality to obtain what is loosely
	\[
	\|D_h - 1\|_{L^p(\mathbb{P})} \leq C' \sqrt{p} \, \|V_h\|_{L^p(\mathbb{P})}^{1/2} \leq 2^{-1/2} C' L \sqrt{p} \, \|J_h\|_{L^p(\mathbb{P})}^{1/2},
	\]
	where $C'$ is universal.
	This prompts us to study $\|J_h\|_{L^p(\mathbb{P})}$ and obtain an upper bound that is of order $h^2(d+p)$ in \cref{lem:integrated-increments}.
	
	To establish the bounds in \cref{lem:integrated-increments}, we first
	control \(X_t-X_0\) using stationary reversibility in
	\cref{lem:linear-increment}, without explicitly solving the
	Langevin equation.

	Once \cref{prop:stationary-rejection} is established, we can utilize elementary techniques in the MCMC literature \citep[see, e.g.,][proof of Theorem~3]{WuSchmidlerChen2022} to establish \cref{prop:overlap}.
	
	\paragraph{Attribution within this appendix.}
	The estimate in \cref{lem:linear-increment} is an elementary consequence
	of the Lyons--Zheng forward--backward decomposition
	\citep[Section~1, especially equation~(1.7)]{LyonsZheng1988}; see also
	\citet[Theorem~5.7.1]{FukushimaOshimaTakeda2011}.
	For the Langevin diffusion considered here, the required decomposition
	follows directly from the stochastic differential equation and stationary
	reversibility, as shown below.
	\cref{lem:integrated-increments} then follows from that estimate,
	Gaussian randomization, and Jensen's inequality.
	The Gaussian-randomization step is the identity-matrix specialization
	of the argument used for sub-Gaussian quadratic forms by
	\citet[Theorem~1 and Remark~2]{HsuKakadeZhang2012}.
	\cref{lem:path-likelihood} combines standard results on stochastic
	exponentials, Novikov's criterion, Girsanov's theorem, stopping of
	stochastic integrals, Doob's inequality, and the quantitative BDG
	inequality, with the problem-specific estimate \(V_h\leq L^2J_h/2\).
	General reverse-H\"older theory for continuous stochastic exponentials
	is developed by \citet{Kazamaki1994}, but the dimension- and
	moment-explicit calculation needed here is supplied in full.
	The passage from the path likelihood to the rejection probability uses
	the standard conditional-likelihood identity and the
	Metropolis--Hastings meet identity
	\citep{Kallenberg2021,Tierney1998,BilleraDiaconis2001}.
	Thus \cref{prop:stationary-rejection} combines these ingredients for
	the present MALA problem.
	
	\subsection{Controlling Langevin increments}
	
	The goal of this subsection is to control the magnitude of \(X_t-X_0\),
	particularly the time integral
	\(\int_0^h\norm{X_t-X_0}^2\,\df t\), where \(X_t\) satisfies
	\eqref{eq:langevin}.
	
	We begin by fixing the continuous-time framework used throughout this
	appendix. All random variables are defined on a filtered probability
	space
	\[
	(\Omega,\mathcal F,(\mathcal F_t)_{t\geq0},\mathbb P)
	\]
	satisfying the usual completeness and right-continuity conditions.
	A standard \(d\)-dimensional Brownian motion relative to this filtration
	is an adapted continuous process \(B\) with \(B_0=0\) such that, for
	\(0\leq s<t\), the increment \(B_t-B_s\) has law
	\(N_d(0,(t-s)I_d)\) and is independent of \(\mathcal F_s\).
	This last independence property will be used explicitly in
	\cref{ssec:girsanov-frozen}.
	Basic references for filtered probability spaces, Brownian motion,
	It\^o integration, and stochastic differential equations are
	\citet[Chapters~3 and~5]{KaratzasShreve1991},
	\citet[Chapters~IV and~IX]{RevuzYor1999}, and
	\citet{Kallenberg2021}.
	
	Because \(\nabla U\) is globally Lipschitz, for every \(x\in\R^d\)
	the Langevin equation has a unique nonexplosive strong solution
	\((X_t^x)_{t\geq0}\) with \(X_0^x=x\); see, for example,
	\citet[Chapter~5]{KaratzasShreve1991}.
	For a bounded measurable function \(f\), define the transition
	semigroup by
	\[
	S_tf(x)=\E[f(X_t^x)].
	\]
	
	Now let \(X_0\) be \(\mathcal F_0\)-measurable and independent of \(B\),
	with law \(\pi\), and let \((X_t)_{t\geq0}\) be the corresponding
	solution of
	\begin{equation}\label{eq:langevin-sde}
		\df X_t=-\nabla U(X_t)\,\df t+\sqrt{2}\,\df B_t,
		\qquad X_0\sim\pi.
	\end{equation}
	The differential notation means the integral identity
	\begin{equation}\label{eq:langevin-integral}
		X_t=X_0-\int_0^t\nabla U(X_s)\,\df s+\sqrt{2}\,B_t,
		\qquad t\geq0.
	\end{equation}
	Under the assumptions stated at the beginning of this appendix,
	the Langevin semigroup is symmetric with respect to
	\(\pi(\df x)\propto e^{-U(x)}\,\df x\):
	\[
	\int fS_tg\,\df\pi=\int gS_tf\,\df\pi
	\]
	for bounded measurable \(f,g\) and \(t\geq0\); see
	\citet[Section~2]{DurmusMoulines2017} and
	\citet[Chapters~1 and~7]{FukushimaOshimaTakeda2011}.
	Since \(S_t1=1\), taking \(f=1\) shows that \(\pi\) is invariant.
	Consequently, the process started from \(X_0\sim\pi\) is stationary
	and reversible.
	
	\begin{lemma}\label{lem:linear-increment}
		For every \(v\in\R^d\), \(\lambda\in\R\), and \(t\geq0\),
		\begin{equation}\label{eq:linear-increment-mgf}
			\E[\exp(\lambda\ip{v}{X_t-X_0})]
			\leq
			\exp(\lambda^2t\norm{v}^2).
		\end{equation}
	\end{lemma}
	
	\begin{proof}
		The case \(t=0\) is immediate. Fix \(t>0\).
		Stationarity and reversibility give
		\[
		(X_s)_{0\leq s\leq t}
		\stackrel{\mathrm d}{=}
		(X_{t-s})_{0\leq s\leq t}
		\]
		as random elements of \(C([0,t];\R^d)\).
		Set
		\[
		Z_t=X_0-X_t+\int_0^t\nabla U(X_s)\,\df s.
		\]
		Reversing the path interchanges \(X_0\) and \(X_t\), while leaving
		the time integral unchanged, since
		\[
		\int_0^t\nabla U(X_{t-s})\,\df s
		=
		\int_0^t\nabla U(X_s)\,\df s.
		\]
		The equality of path laws and \eqref{eq:langevin-integral} therefore give
		\[
		Z_t
		\stackrel{\mathrm d}{=}
		X_t-X_0+\int_0^t\nabla U(X_s)\,\df s
		=
		\sqrt{2}\,B_t
		\sim N_d(0,2tI_d).
		\]
		Moreover, subtracting the definition of \(Z_t\) from the preceding
		expression for \(\sqrt{2}\,B_t\) yields
		\[
		X_t-X_0=\frac{\sqrt{2}\,B_t-Z_t}{2}.
		\]
		Thus the drift integral cancels when the forward and reversed
		expressions are subtracted.
		
		By convexity of the exponential function,
		\begin{align*}
			\E[\exp(\lambda\ip{v}{X_t-X_0})]
			&=
			\E\left[
			\exp\left(\frac{\lambda}{2}
			\ip{v}{\sqrt{2}\,B_t-Z_t}\right)
			\right]\\
			&\leq
			\frac12\E[\exp(\sqrt{2}\lambda\ip{v}{B_t})]
			+
			\frac12\E[\exp(-\lambda\ip{v}{Z_t})]\\
			&=
			\exp(\lambda^2t\norm{v}^2).
		\end{align*}
		The last equality uses the Gaussian moment-generating formula.
		No independence between \(B_t\) and \(Z_t\) is required, and the
		finite Gaussian exponential moments justify the inequality without
		any preliminary integrability assumption on its left-hand side.
	\end{proof}
	
	The next proof uses a standard Gaussian-randomization argument.  In a more
	general matrix form, it is the argument used by
	\citet[Theorem~1 and Remark~2]{HsuKakadeZhang2012} to pass from linear
	sub-Gaussian moment bounds to quadratic-form bounds.  We include the
	identity-matrix calculation because the exact constant in
	\eqref{eq:quadratic-increment-mgf} is useful below.
	
	\begin{lemma} \label{lem:integrated-increments}
		For every \(t>0\) and \(0\leq a<1/(4t)\),
		\begin{equation}\label{eq:quadratic-increment-mgf}
			\E[\exp(a\norm{X_t-X_0}^2)]
			\leq
			(1-4at)^{-d/2}.
		\end{equation}
		If
		\[
		J_h=\int_0^h\norm{X_t-X_0}^2\,\df t,
		\]
		then, for \(0\leq\lambda<1/(4h^2)\),
		\begin{equation}\label{eq:Jh-mgf}
			\E (e^{\lambda J_h})
			\leq
			(1-4\lambda h^2)^{-d/2}.
		\end{equation}
		Consequently, for every real \(p\geq1\),
		\begin{equation}\label{eq:Jh-moment}
			\norm{J_h}_{L^p(\mathbb P)}
			\leq
			C h^2(d+p),
		\end{equation}
		where \(C\) is universal.
	\end{lemma}
	
	\begin{proof}
		Enlarge the probability space,
		if necessary, so that it also carries a random vector
		\(G\sim N_d(0,I_d)\) independent of the entire Langevin path
		\((X_s)_{s\geq0}\).  For a fixed vector \(z\in\R^d\), the scalar
		\(\ip{G}{z}\) is Gaussian with mean zero and variance \(\norm{z}^2\).
		The Gaussian moment-generating formula therefore gives
		\begin{equation*}
			\E\!
			\left(
			e^{\sqrt{2a}\ip{G}{z}}
			\right)
			=e^{a\norm{z}^2}.
		\end{equation*}
		Apply this identity with the random vector \(z=X_t-X_0\).  More precisely,
		conditioning first on the Langevin path gives
		\begin{align*}
			\E \left(e^{a\norm{X_t-X_0}^2} \right)
			&=\E\left\{
			\E\left[
			e^{\sqrt{2a}\ip{G}{X_t-X_0}}
			\mathrel{\big|}\sigma(X_s:s\geq0)
			\right]
			\right\}.
		\end{align*}
		The integrand is nonnegative, so Tonelli's theorem allows us to reverse the
		order in which the path and \(G\) are averaged.  Conditional on \(G\), the
		vector \(G\) is fixed, and \cref{lem:linear-increment}, with
		\(v=G\) and \(\lambda=\sqrt{2a}\), gives
		\[
		\E\left(
		e^{\sqrt{2a}\ip{G}{X_t-X_0}}
		\mathrel{\big|}G
		\right)
		\leq e^{2at\norm{G}^2}.
		\]
		Consequently,
		\begin{equation*}
			\begin{aligned}
				\E \left(e^{a\norm{X_t-X_0}^2} \right)
				&\leq \E \left(e^{2at\norm{G}^2} \right)\\
				&=(1-4at)^{-d/2},
			\end{aligned}
		\end{equation*}
		where the last equality is the moment-generating function of a
		\(\chi_d^2\) random variable and requires \(4at<1\).  This proves
		\eqref{eq:quadratic-increment-mgf}.
		
		We next pass from one increment to its time integral.  For each fixed sample
		path, Jensen's inequality for the convex function \(x\mapsto e^x\), applied
		to the uniform probability measure \(\df t/h\) on \([0,h]\), gives
		\begin{align*}
			e^{\lambda J_h}
			&=\exp\left(
			\frac1h\int_0^h \lambda h\norm{X_t-X_0}^2\,\df t
			\right)\\
			&\leq
			\frac1h\int_0^h
			e^{\lambda h\norm{X_t-X_0}^2}\,\df t.
		\end{align*}
		Take expectations and use Tonelli's theorem, followed by
		\eqref{eq:quadratic-increment-mgf} with \(a=\lambda h\):
		\begin{align*}
			\E (e^{\lambda J_h})
			&\leq
			\frac1h\int_0^h
			\E (e^{\lambda h\norm{X_t-X_0}^2}) \,\df t\\
			&\leq
			\frac1h\int_0^h(1-4\lambda ht)^{-d/2}\,\df t\\
			&\leq(1-4\lambda h^2)^{-d/2}.
		\end{align*}
		This proves \eqref{eq:Jh-mgf}.
		
		It remains to derive the \(L^p\) estimate from this exponential-tail bound.
		Put \(\lambda=(8h^2)^{-1}\) in
		\eqref{eq:Jh-mgf}.  Markov's inequality gives, for every \(s\geq0\),
		\begin{equation*}
			\mathbb P\left(
			J_h\geq4dh^2\log 2+8h^2s
			\right)
			\leq e^{-s}.
		\end{equation*}
		Define
		\[
		Y=\frac{(J_h-4dh^2\log 2)_+}{8h^2}.
		\]
		Then \(\mathbb P(Y>s)\leq e^{-s}\).  For every real \(p\geq1\), the
		layer-cake formula for moments gives
		\begin{align*}
			\E (Y^p)
			&=p\int_0^\infty s^{p-1}\mathbb P(Y>s)\,\df s\\
			&\leq p\int_0^\infty s^{p-1}e^{-s}\,\df s
			=\Gamma(p+1).
		\end{align*}
		Stirling's bound implies \(\Gamma(p+1)^{1/p}\leq Cp\) for all
		\(p\geq1\).  Hence, by the Minkowski inequality,
		\[
		\norm{J_h}_{L^p(\mathbb P)}
		\leq4dh^2\log 2+8h^2\norm{Y}_{L^p(\mathbb P)}
		\leq Ch^2(d+p),
		\]
		which proves \eqref{eq:Jh-moment}.  For the tail-integration identity and the
		Gamma-function estimate, see
		\citet[Chapter~2]{BoucheronLugosiMassart2013}.
	\end{proof}
	
	\subsection{Girsanov comparison with frozen drift}\label{ssec:girsanov-frozen}
	
	We compare two laws on the same path space over \([0,h]\).  Under the
	probability measure \(\mathbb P\), the path satisfies the exact
	Langevin equation \eqref{eq:langevin-sde}.  We shall construct a second
	measure under which the coordinate process has drift frozen at the initial
	value \(-\nabla U(X_0)\).
	
	We recall the continuous-martingale facts used below.  If \(H_t\) is a
	predictable \(\R^d\)-valued process with
	\(\int_0^h\norm{H_t}^2\,\df t<\infty\) almost surely, then
	\(\int_0^t\ip{H_s}{\df B_s}\) is a continuous local martingale.  Its quadratic
	variation is
	\[
	\left[
	\int_0^\cdot\ip{H_s}{\df B_s}
	\right]_t
	=\int_0^t\norm{H_s}^2\,\df s.
	\]
	For a continuous
	local martingale \(N\) with \(N_0=0\), the stochastic exponential is
	\[
	\mathcal E(N)_t=\exp(N_t-[N]_t/2).
	\]
	It is always a nonnegative local martingale, but it need not have expectation
	one.  Novikov's criterion states that
	\(\E\exp([N]_h/2)<\infty\) is sufficient for
	\((\mathcal E(N)_t)_{0\leq t\leq h}\) to be a true mean-one martingale.
	These statements, together with It\^o's formula and the stopping arguments
	used below, are standard; see \citet[Chapter~3]{KaratzasShreve1991},
	\citet[Chapters~IV and~VIII]{RevuzYor1999}, or
	\citet{Kallenberg2021}.
	
	Define
	\begin{equation*}
		\theta_t=
		\frac{\nabla U(X_t)-\nabla U(X_0)}{\sqrt{2}},
		\qquad
		V_t=\int_0^t\norm{\theta_s}^2\,\df s,
		\qquad
		M_t=\int_0^t\ip{\theta_s}{\df B_s}.
	\end{equation*}
	The process \(\theta_t\) is continuous and adapted, hence predictable.  Since
	\(X\) has continuous paths, \(\int_0^h\norm{\theta_s}^2\,\df s<\infty\) almost
	surely on every finite interval, so the It\^o integral $M_t$ is well defined as a
	continuous local \(\mathbb P\)-martingale.  Its quadratic variation is
	\([M]_t=V_t\).  The Lipschitz property of \(\nabla U\) and
	\cref{lem:integrated-increments} give
	\begin{equation}\label{eq:V-J}
		V_h
		\leq
		\frac{L^2}{2}J_h,
		\qquad
		\norm{V_h}_{L^p(\mathbb P)}
		\leq
		C L^2h^2(d+p)
		\; \text{whenever }p\geq1.
	\end{equation}
	Define the stochastic exponential
	\begin{equation*}
		D_t=\exp\left( M_t-\frac12V_t\right) = \mathcal E(M)_t.
	\end{equation*}
	
	The following lemma collects the change-of-measure construction that will be used in the proof of
	\cref{prop:stationary-rejection}.
	
	\begin{lemma}
		\label{lem:frozen-endpoint-law}
		Assume that \(h>0\) and \(Lh<1\).  Then
		\((D_t)_{0\leq t\leq h}\) is a true mean-one \(\mathbb P\)-martingale, and
		\begin{equation}\label{eq:frozen-measure-definition}
			\frac{\df\mathbb P_h^{\mathrm{fr}}}{\df\mathbb P}
			\bigg|_{\mathcal F_h}=D_h
		\end{equation}
		defines a probability measure \(\mathbb P_h^{\mathrm{fr}}\) on \((\Omega,\mathcal F_h)\).  Under
		\(\mathbb P_h^{\mathrm{fr}}\), the process
		\[
		W_t=B_t-\int_0^t\theta_s\,\df s,
		\qquad 0\leq t\leq h,
		\]
		is a Brownian motion relative to
		\((\mathcal F_t)_{0\leq t\leq h}\), and
		\begin{equation}\label{eq:frozen-path}
			X_t=X_0-t\nabla U(X_0)+\sqrt{2}\,W_t,
			\qquad 0\leq t\leq h.
		\end{equation}
		Moreover, under \(\mathbb P_h^{\mathrm{fr}}\), \(X_0\sim\pi\), and a version
		of the conditional law of \(X_h\) given \(X_0=x\) is \(Q_h(x,\cdot)\).
		It follows that, for all Borel sets \(A,B\subseteq\R^d\),
		\begin{equation}\label{eq:frozen-endpoint-law}
			\mathbb P_h^{\mathrm{fr}}(X_0\in A,X_h\in B)
			=\int_A Q_h(x,B)\,\pi(\df x).
		\end{equation}
	\end{lemma}
	
	\begin{proof}
		By \eqref{eq:Jh-mgf} and \eqref{eq:V-J},
		\begin{align*}
			\E\left[\exp\left([M]_h/2\right)\right]
			&=\E\left[\exp\left(V_h/2\right)\right]\\
			&\leq\E\left[\exp\left(L^2J_h/4\right)\right]\\
			&\leq(1-L^2h^2)^{-d/2}<\infty,
		\end{align*}
		where the last inequality uses \(Lh<1\).  Novikov's criterion therefore
		shows that \(D\) is a true mean-one \(\mathbb P\)-martingale.  In
		particular,
		\begin{equation}\label{eq:D-conditional-martingale}
			\E(D_t\mid\mathcal F_s)=D_s,
			\qquad 0\leq s\leq t\leq h,
		\end{equation}
		and \eqref{eq:frozen-measure-definition} defines a probability measure.
		
		Because $D$ is a true martingale, Girsanov's theorem gives the asserted Brownian motion \(W_t\); see
		\citet[Section~3.5]{KaratzasShreve1991} or
		\citet[Chapter~VIII, Section~1]{RevuzYor1999}.  Substituting
		\[
		B_t=W_t+\int_0^t\theta_s\,\df s
		\]
		into \eqref{eq:langevin-integral} yields
		\begin{align*}
			X_t
			&=X_0-\int_0^t\nabla U(X_s)\,\df s
			+\sqrt{2}\,W_t
			+\sqrt{2}\int_0^t\theta_s\,\df s\\
			&=X_0-t\nabla U(X_0)+\sqrt{2}\,W_t,
		\end{align*}
		which proves \eqref{eq:frozen-path}.
		
		For every Borel set \(A\subseteq\R^d\),
		\begin{align*}
			\mathbb P_h^{\mathrm{fr}}(X_0\in A)
			&=\E\left(D_h\1_{\{X_0\in A\}}\right)\\
			&=\E\left[
			\E(D_h\mid\mathcal F_0)\1_{\{X_0\in A\}}
			\right]\\
			&=\mathbb P(X_0\in A)
			=\pi(A),
		\end{align*}
		where \eqref{eq:D-conditional-martingale} gives
		\(\E[D_h\mid\mathcal F_0]=D_0=1\).
		
		Under \(\mathbb P_h^{\mathrm{fr}}\), the increment \(W_h-W_0=W_h\) is
		independent of \(\mathcal F_0\) and has law \(N_d(0,hI_d)\).  Since \(X_0\)
		is \(\mathcal F_0\)-measurable, \eqref{eq:frozen-path} gives, conditionally
		on \(X_0=x\),
		\[
		X_h=x-h\nabla U(x)+\sqrt{2h}\,Z,
		\qquad Z\sim N_d(0,I_d).
		\]
		This is precisely the kernel \(Q_h(x,\cdot)\).  Combining this conditional
		law with \(X_0\sim\pi\) proves \eqref{eq:frozen-endpoint-law}.
	\end{proof}
	
	We next quantify how close the path likelihood \(D_h\) is to one.
	General reverse-H\"older and moment estimates for continuous stochastic
	exponentials form a classical theory; see, for example,
	\citet[Chapters~1 and~3]{Kazamaki1994}.  We give a direct localized
	calculation to track explicitly the dependence on $d,p,L$, and $h$ and to
	exploit the problem-specific bound $V_h\leq L^2J_h/2$.
	
	\begin{lemma}\label{lem:path-likelihood}
		There are universal constants \(c,C>0\) such that, for every \(h>0\) and
		every real \(p\geq 1\), the conditions
		\begin{equation}\label{eq:D-moment-condition}
			Lh<1,
			\qquad
			h\leq\frac{c}{L\sqrt{p(d+p)}}
		\end{equation}
		imply
		\begin{equation}\label{eq:D-moments}
			\norm{D_h}_{L^{2p}(\mathbb P)}\leq C,
			\qquad
			\norm{D_h-1}_{L^p(\mathbb P)}
			\leq CLh\sqrt{p(d+p)}.
		\end{equation}
	\end{lemma}
	
	\begin{proof}
		\noindent {\it Stopping conventions.}
		We use standard process-level stopping identities. If \(N_t\) is a continuous local martingale and
		\(\tau\) is a stopping time, then, up to indistinguishability,
		\begin{equation} \label{eq:quadratic-stop}
			N_t^\tau:=N_{t\wedge\tau},
			\qquad
			[N^\tau]_t=[N]_{t\wedge\tau}.
		\end{equation}
		Likewise, for a predictable integrand \(H_t\),
		\begin{equation}\label{eq:stopped-stochastic-integral}
			\left(\int_0^\cdot\ip{H_s}{\df B_s}\right)_{t\wedge\tau}
			=
			\int_0^t\1_{\{s\leq\tau\}}\ip{H_s}{\df B_s}.
		\end{equation}
		These are identities of stopped processes and therefore hold for all
		\(t\) simultaneously outside a single null set; continuity then permits
		evaluation at any bounded random time. Purely algebraic identities among
		continuous processes may be evaluated at \(t\wedge\tau\) for the same
		reason, and inequalities such as
		\(V_{t\wedge\tau}\leq V_t\) follow pathwise from monotonicity. By
		contrast, an estimate obtained only after taking expectations at each fixed
		time cannot generally be evaluated at a stopping time without a separate
		optional-stopping argument. No such step is used here.
		See \citet[Chapter~IV, Sections~1--2]{RevuzYor1999},
		\citet[Chapter~3]{KaratzasShreve1991}, or
		\citet[Chapter~II]{Protter2005} for stopping of continuous local
		martingales, stochastic integrals, and quadratic variation.
		
		\smallskip
		\noindent\emph{Step 1: a uniform \(L^{2p}\) bound for \(D_h\).}
		For \(n\geq1\), let
		\[
		\tau_n=\inf\{t\in[0,h]:V_t\geq n\},
		\qquad
		\inf\varnothing=\infty.
		\]
		Because \(V\) is continuous and adapted, \(\tau_n\) is a stopping time.
		Set \(M_t^{(n)}=M_{t\wedge\tau_n}\). By \eqref{eq:quadratic-stop},
		\[
		[M^{(n)}]_t=[M]_{t\wedge\tau_n}=V_{t\wedge\tau_n}.
		\]
		Consequently, the quadratic variation of the continuous local martingale
		\(4pM^{(n)}\) is
		\[
		[4pM^{(n)}]_t
		=
		16p^2V_{t\wedge\tau_n}
		\leq16p^2n,
		\qquad 0\leq t\leq h.
		\]
		Consider its stochastic exponential
		\[
		\Lambda_t^{(2p,n)}
		=
		\exp\left(
		4pM_{t\wedge\tau_n}-8p^2V_{t\wedge\tau_n}
		\right).
		\]
		Novikov's criterion applies because
		\[
		\E\left[
		\exp\left(\frac12[4pM^{(n)}]_h\right)
		\right]
		\leq e^{8p^2n}<\infty.
		\]
		Hence \(\Lambda^{(2p,n)}\) is a true mean-one martingale, and
		\begin{equation}\label{eq:Lambda-mean-one}
			\E\left[\Lambda_h^{(2p,n)}\right]=1.
		\end{equation}
		
		The definitions of \(D\) and \(\Lambda^{(2p,n)}\) give the exact algebraic
		identity
		\[
		D_{h\wedge\tau_n}^{2p}
		=
		\bigl[\Lambda_h^{(2p,n)}\bigr]^{1/2}
		e^{(4p^2-p)V_{h\wedge\tau_n}}.
		\]
		Cauchy--Schwarz, \eqref{eq:Lambda-mean-one}, and the inequalities
		\(V_{h\wedge\tau_n}\leq V_h\) and \(8p^2-2p\geq0\), valid for \(p\geq1\), yield
		\begin{align*}
			\E\left[D_{h\wedge\tau_n}^{2p}\right]
			&\leq
			\bigl\{\E[\Lambda_h^{(2p,n)}]\bigr\}^{1/2}
			\left\{
			\E\left[e^{(8p^2-2p)V_{h\wedge\tau_n}}\right]
			\right\}^{1/2}\\
			&\leq
			\left\{
			\E\left[e^{(8p^2-2p)V_h}\right]
			\right\}^{1/2}.
		\end{align*}
		Since \(V_h\leq L^2J_h/2\), as noted in \eqref{eq:V-J}, equation \eqref{eq:Jh-mgf} gives
		\[
		\E\left[D_{h\wedge\tau_n}^{2p}\right]
		\leq
		(1-u)^{-d/4},
		\qquad
		u=4p(4p-1)L^2h^2,
		\]
		provided \(u<1\). By \eqref{eq:D-moment-condition},
		\[
		u
		\leq16p^2L^2h^2
		\leq16c^2\frac{p}{d+p}
		\leq16c^2.
		\]
		Choose the universal constant \(c>0\) in the statement so that
		\(16c^2\leq1/2\). It follows that \(0\leq u\leq1/2\), and hence
		\[
		\log\left(\frac1{1-u}\right)\leq2u.
		\]
		Therefore
		\begin{align*}
			\log\norm{D_{h\wedge\tau_n}}_{L^{2p}(\mathbb P)}
			&=
			\frac1{2p}
			\log\left(
			\E\left[D_{h\wedge\tau_n}^{2p}\right]
			\right)\\
			&\leq
			\frac{d}{8p}\log\left(\frac1{1-u}\right)\\
			&\leq\frac{du}{4p}\\
			&\leq4dpL^2h^2\\
			&\leq4c^2\frac{d}{d+p}\\
			&\leq4c^2.
		\end{align*}
		Thus
		\begin{equation}\label{eq:D-stopped-uniform}
			\norm{D_{h\wedge\tau_n}}_{L^{2p}(\mathbb P)}
			\leq e^{4c^2}
		\end{equation}
		uniformly in \(n\). Finally, \(V_h<\infty\) almost surely. For every
		sample path, once \(n>V_h\), one has \(\tau_n=\infty\) and hence
		\(D_{h\wedge\tau_n}=D_h\). Fatou's lemma applied to the \((2p)\)-th
		powers in \eqref{eq:D-stopped-uniform} gives
		\begin{equation}\label{eq:D-2p-uniform}
			\norm{D_h}_{L^{2p}(\mathbb P)}
			\leq e^{4c^2}.
		\end{equation}
		This proves the first estimate in \eqref{eq:D-moments}.
		
		\smallskip
		\noindent\emph{Step 2: representation of \(D_h-1\) as a stochastic integral.}
		Let
		\[
		Y_t=M_t-\frac12V_t,
		\qquad
		D_t=e^{Y_t}.
		\]
		Since \(M\) is a continuous local \(\mathbb P\)-martingale and
		\(V\) is continuous and of finite variation, \(Y\) is a continuous
		semimartingale with quadratic variation
		\([Y]_t=[M]_t=V_t\).
		Applying It\^o's formula for continuous semimartingales
		\citep[Section~3.3.A, Theorem~3.3]{KaratzasShreve1991}
		to \(y\mapsto e^y\) gives
		\begin{align*}
			\df D_t
			&=
			D_t\,\df Y_t+\frac12D_t\,\df[Y]_t\\
			&=
			D_t\left(\df M_t-\frac12\df V_t\right)
			+\frac12D_t\,\df V_t\\
			&=
			D_t\,\df M_t\\
			&=
			D_t\ip{\theta_t}{\df B_t}.
		\end{align*}
		Integrating this identity from zero to \(t\) gives the local-martingale
		representation
		\begin{equation}\label{eq:D-minus-one-martingale}
			D_t-1
			=
			\int_0^tD_s\ip{\theta_s}{\df B_s}.
		\end{equation}
		See \citet[Chapter~3]{KaratzasShreve1991} or
		\citet[Chapter~IV]{RevuzYor1999} for It\^o's formula and the quadratic
		variation rules used here.
		
		\smallskip
		\noindent\emph{Step 3: application of BDG to an explicitly stopped martingale.}
		For \(n\geq2\), define
		\[
		\sigma_n
		=
		\tau_n\wedge
		\inf\{t\in[0,h]:D_t\geq n\},
		\qquad
		\inf\varnothing=\infty.
		\]
		Both \(V_t\) and \(D_t\) are continuous and adapted, so \(\sigma_n\) is a
		stopping time. Stopping the process identity
		\eqref{eq:D-minus-one-martingale} and using
		\eqref{eq:stopped-stochastic-integral} gives, for every deterministic
		\(t\in[0,h]\),
		\begin{align*}
			R_t^{(n)}
			&:=
			D_{t\wedge\sigma_n}-1\\
			&=
			\int_0^{t\wedge\sigma_n}
			D_s\ip{\theta_s}{\df B_s}\\
			&=
			\int_0^t
			\1_{\{s\leq\sigma_n\}}D_s\ip{\theta_s}{\df B_s}.
		\end{align*}
		Thus \(R^{(n)}\) is a square-integrable continuous martingale. Indeed,
		its quadratic variation is
		\[
		[R^{(n)}]_h
		=
		\int_0^h
		\1_{\{s\leq\sigma_n\}}
		D_s^2\norm{\theta_s}^2\,\df s
		=
		\int_0^{h\wedge\sigma_n}
		D_s^2\norm{\theta_s}^2\,\df s
		\leq
		n^2V_{h\wedge\tau_n}
		\leq n^3.
		\]
		The possible distinction between \(s<\sigma_n\) and
		\(s\leq\sigma_n\) is immaterial for the Lebesgue and It\^o integrals.
		
		Apply the upper inequality in
		\citet[Theorem~2, equation~(23)]{Ren2008}, valid for every real \(p\geq1\),
		to the stopped continuous martingale \(R^{(n)}\) at time \(h\).
		This yields
		\begin{equation}\label{eq:BDG-stopped}
			\norm{R_h^{(n)}}_{L^p(\mathbb P)}
			\leq
			\norm{\sup_{0\leq t\leq h}\abs{R_t^{(n)}}}_{L^p(\mathbb P)}
			\leq
			2\sqrt{2p}\,
			\norm{[R^{(n)}]_h^{1/2}}_{L^p(\mathbb P)}.
		\end{equation}
		
		Pathwise,
		\[
		[R^{(n)}]_h^{1/2}
		\leq
		\left(\sup_{0\leq s\leq h}D_s\right)V_h^{1/2}.
		\]
		H\"older's inequality with exponents two and two therefore gives
		\begin{align}
			\norm{[R^{(n)}]_h^{1/2}}_{L^p(\mathbb P)}
			&\leq
			\norm{\sup_{0\leq s\leq h}D_s}_{L^{2p}(\mathbb P)}
			\norm{V_h^{1/2}}_{L^{2p}(\mathbb P)}
			\notag\\
			&=
			\norm{\sup_{0\leq s\leq h}D_s}_{L^{2p}(\mathbb P)}
			\norm{V_h}_{L^p(\mathbb P)}^{1/2}.
			\label{eq:BDG-holder-details}
		\end{align}
		By \cref{lem:frozen-endpoint-law}, since $Lh < 1$, \((D_t)_{0 \leq t \leq h}\) is a nonnegative true
		martingale. 
		Conditional Jensen's inequality gives
		\[
		\sup_{0\leq s\leq h}\norm{D_s}_{L^{2p}(\mathbb P)}
		=\norm{D_h}_{L^{2p}(\mathbb P)}.
		\]
		Since \(2p > 1\), Doob's maximal inequality
		\citep[Chapter~II, Theorem~1.7]{RevuzYor1999} gives
		\begin{align*}
			\norm{\sup_{0\leq s\leq h}D_s}_{L^{2p}(\mathbb P)}
			&\leq\frac{2p}{2p-1}
			\sup_{0\leq s\leq h}\norm{D_s}_{L^{2p}(\mathbb P)}\\
			&=\frac{2p}{2p-1}\norm{D_h}_{L^{2p}(\mathbb P)}
			\leq2e^{4c^2},
		\end{align*}
		where the last inequality uses \eqref{eq:D-2p-uniform} and
		\(2p/(2p-1)\leq2\) for \(p\geq1\). Moreover,
		\eqref{eq:V-J} gives
		\[
		\norm{V_h}_{L^p(\mathbb P)}^{1/2}
		\leq CLh\sqrt{d+p}.
		\]
		Substituting these bounds into
		\eqref{eq:BDG-stopped}--\eqref{eq:BDG-holder-details} gives, uniformly in
		\(n\),
		\begin{equation}\label{eq:D-minus-one-stopped}
			\norm{D_{h\wedge\sigma_n}-1}_{L^p(\mathbb P)}
			\leq CLh\sqrt{p(d+p)}.
		\end{equation}
		To remove the stopping, note that the continuous functions
		\(t\mapsto V_t\) and \(t\mapsto D_t\) are bounded on the compact interval
		\([0,h]\) for every sample path. Hence \(\sigma_n>h\) for all sufficiently
		large \(n\), path by path, and
		\[
		D_{h\wedge\sigma_n}-1\longrightarrow D_h-1
		\qquad\text{almost surely}.
		\]
		Fatou's lemma applied to \eqref{eq:D-minus-one-stopped} proves
		\[
		\norm{D_h-1}_{L^p(\mathbb P)}
		\leq CLh\sqrt{p(d+p)}.
		\]
		This is the second estimate in \eqref{eq:D-moments} and completes the proof.
	\end{proof}

	\subsection{From path likelihood to MALA rejection}
	
	\begin{proof}[Proof of \cref{prop:stationary-rejection}]
		
		Choose the universal constant \(c\) in
		\cref{prop:stationary-rejection} no larger than the constant~$c$ in
		\cref{lem:path-likelihood} and sufficiently small so that the step-size
		condition $h \leq c/[L\sqrt{p(d+p)}]$ implies \(Lh<1\).
		This will allow us to apply \cref{lem:frozen-endpoint-law,lem:path-likelihood}.
		
		Let \(S_h\) be the exact Langevin transition kernel.  Define two endpoint
		probability measures on \(\R^d\times\R^d\):
		\begin{equation*}
			\nu_h(\df x,\df y)=\pi(\df x)S_h(x,\df y),
			\qquad
			\mu_h(\df x,\df y)=\pi(\df x)Q_h(x,\df y).
		\end{equation*}
		The measure \(\nu_h\) is the law of \((X_0,X_h)\) under the exact path law
		\(\mathbb P\).  By \cref{lem:frozen-endpoint-law}, \(\mu_h\) is the law of
		\((X_0,X_h)\) under \(\mathbb P_h^{\mathrm{fr}}\).  Moreover,
		\eqref{eq:frozen-measure-definition} gives
		\[
		\frac{\df\mathbb P_h^{\mathrm{fr}}}{\df\mathbb P}=D_h
		\quad\text{on }\mathcal F_h,
		\]
		so \(\mu_h\ll\nu_h\).  Let
		\[
		F_h=\frac{\df\mu_h}{\df\nu_h}.
		\]
		
		We next pass from the full-path likelihood \(D_h\) to the endpoint likelihood
		\(F_h\).  By the Doob--Dynkin lemma,
		\(\E[D_h\mid\sigma(X_0,X_h)]\) has a Borel-measurable representation as a
		function of the two endpoints; this function is a version of \(F_h\).  Indeed,
		for every bounded Borel function
		\(\varphi:\R^d\times\R^d\to\R\),
		\begin{align*}
			\int\varphi(x,y)\,\mu_h(\df x,\df y)
			&=\E_{\mathbb P_h^{\mathrm{fr}}}[\varphi(X_0,X_h)]\\
			&=\E[D_h\varphi(X_0,X_h)]\\
			&=\E\left\{
			\E[D_h\mid\sigma(X_0,X_h)]
			\varphi(X_0,X_h)
			\right\}.
		\end{align*}
		This is exactly the defining Radon--Nikodym identity.  Thus
		\begin{equation}\label{eq:endpoint-conditioning}
			F_h(X_0,X_h)
			=\E[D_h\mid\sigma(X_0,X_h)]
			\qquad \mathbb P\text{-almost surely}.
		\end{equation}
		This is the usual principle that the likelihood ratio of an observed
		statistic is the conditional expectation of the full-data likelihood ratio;
		see \citet{Kallenberg2021} for conditional expectation, disintegration, and
		regular conditional distributions.
		
		Subtracting one in \eqref{eq:endpoint-conditioning} and using conditional
		Jensen's inequality gives the \(L^p\)-contraction explicitly:
		\begin{align*}
			\norm{F_h-1}_{L^p(\nu_h)}^p
			&=\E\left|
			\E[D_h-1\mid\sigma(X_0,X_h)]
			\right|^p\\
			&\leq\E\left\{
			\E\left[
			|D_h-1|^p\mid\sigma(X_0,X_h)
			\right]
			\right\} \\
			&=\norm{D_h-1}_{L^p(\mathbb P)}^p.
		\end{align*}
		Hence
		\begin{equation}\label{eq:F-contraction}
			\norm{F_h-1}_{L^p(\nu_h)}
			\leq
			\norm{D_h-1}_{L^p(\mathbb P)}.
		\end{equation}
		
		We also need the conditional form of \(\mu_h=F_h\nu_h\).  Evaluating this
		measure identity on a rectangle \(A\times B\) gives
		\begin{equation}\label{eq:rectangle-disintegration}
			\int_AQ_h(x,B)\,\pi(\df x)
			=\int_A\int_BF_h(x,y)S_h(x,\df y)\,\pi(\df x).
		\end{equation}
		Fix a countable \(\pi\)-system \(\mathcal C\) that generates the Borel
		\(\sigma\)-field on \(\R^d\), for example, the finite intersections of open
		balls with rational centers and radii, together with \(\R^d\).  For each
		\(B\in\mathcal C\), equality
		\eqref{eq:rectangle-disintegration} for every Borel \(A\) implies equality
		of the two integrands for \(\pi\)-almost every \(x\), i.e.,
		\[
		Q_h(x,B) = \int_B F_h(x,y) S_h(x, \df y).
		\]
		Since
		\(\mathcal C\) is countable, there is one \(\pi\)-null set outside which
		this holds simultaneously for every \(B\in\mathcal C\).  For each fixed
		\(x\) outside that null set, both sides are measures in \(B\), and the
		monotone-class theorem extends their equality from \(\mathcal C\) to all
		Borel sets.  Thus, for \(\pi\)-almost every \(x\),
		\begin{equation}\label{eq:conditional-F}
			Q_h(x,\df y)=F_h(x,y)S_h(x,\df y).
		\end{equation}
		
		By reversibility, \(\nu_h\) is symmetric. Hence the transpose
		measure \(\mu_h^\top(\df x,\df y)=\mu_h(\df y,\df x)\)
		has density \(F_h(y,x)\) with respect to \(\nu_h\).
		Since the target and proposal densities are strictly positive,
		the Radon--Nikodym chain rule gives
		\[
		F_h(y,x)
		=
		\frac{e^{-U(y)}q_h(y,x)}
		{e^{-U(x)}q_h(x,y)}
		F_h(x,y)
		\qquad \nu_h\text{-almost everywhere}.
		\]
		The definition of the Metropolis--Hastings acceptance probability
		therefore implies
		\[
		F_h(x,y)\alpha_h(x,y)
		=
		\min\{F_h(x,y),F_h(y,x)\}
		\qquad \nu_h\text{-almost everywhere}.
		\]
		This is the usual Metropolis--Hastings minimum identity
		\citep{Tierney1998,BilleraDiaconis2001}.
		
		Because \(\nu_h(\df x,\df y)=\pi(\df x)S_h(x,\df y)\),
		Tonelli's theorem shows that the preceding identity holds for
		\(S_h(x,\cdot)\)-almost every \(y\), for \(\pi\)-almost every \(x\).
		Using \eqref{eq:conditional-F}, we obtain,
		for \(\pi\)-almost every \(x\),
		\begin{align*}
			r_h(x)
			&= 1 - \int \alpha_h(x,y) \, Q_h(x, \df y) \\
			&=1 - \int \alpha_h(x,y) \, F_h(x,y) \, S_h(x, \df y) \\
			&=\int
			\bigl[1-\min\{F_h(x,y),F_h(y,x)\}\bigr]
			\,S_h(x,\df y)\\
			&\leq
			\int
			\bigl(\abs{F_h(x,y)-1}+\abs{F_h(y,x)-1}\bigr)
			\,S_h(x,\df y).
		\end{align*}
		
		Define
		\[
		A_h(x)=\int\abs{F_h(x,y)-1}S_h(x,\df y),
		\qquad
		B_h(x)=\int\abs{F_h(y,x)-1}S_h(x,\df y).
		\]
		Then \(r_h\leq A_h+B_h\).  Jensen's inequality for the probability measure
		\(S_h(x,\df y)\) gives
		\begin{align*}
			\norm{A_h}_{L^p(\pi)}^p
			&\leq\int\abs{F_h(x,y)-1}^p\,\nu_h(\df x,\df y),\\
			\norm{B_h}_{L^p(\pi)}^p
			&\leq\int\abs{F_h(y,x)-1}^p\,\nu_h(\df x,\df y).
		\end{align*}
		The first right-hand side is
		\(\norm{F_h-1}_{L^p(\nu_h)}^p\); the second is the same because \(\nu_h\)
		is symmetric.  Minkowski's inequality therefore yields
		\begin{equation}\label{eq:rejection-F}
			\norm{r_h}_{L^p(\pi)}
			\leq
			\norm{A_h}_{L^p(\pi)}+\norm{B_h}_{L^p(\pi)}
			\leq2\norm{F_h-1}_{L^p(\nu_h)}.
		\end{equation}
		Combining \eqref{eq:rejection-F}, \eqref{eq:F-contraction}, and
		\cref{lem:path-likelihood} gives
		\[
		\norm{r_h}_{L^p(\pi)}
		\leq 2\norm{D_h-1}_{L^p(\mathbb P)}
		\leq CLh\sqrt{p(d+p)},
		\]
		after absorbing the factor two into the universal constant.  This is
		\eqref{eq:stationary-rejection}.
	\end{proof}
	
	\subsection{Proving \cref{prop:overlap}} \label{app:overlap}
	
	\begin{proof}
		We begin with the first of the two assertions.
		Let \(c\) be the step-size constant in \cref{prop:stationary-rejection},
		and choose \(0<c_r\leq\min\{c,1\}\).
		Fix $p \geq 1$ and $t$ satisfying \eqref{eq:overlap-step-condition}.
		Then \cref{prop:stationary-rejection} applies to every
		\(h\in[t/2,t]\), and
		\[
		Lt\leq\frac{c_r}{\sqrt{p(d+p)}}\leq\frac1{\sqrt2}<2.
		\]
		Define the rejection probability averaged over \([t/2,t]\) by
		\[
		\overline r_t(x)=\frac2t\int_{t/2}^t r_h(x)\,\df h.
		\]
		Minkowski's integral inequality and \cref{prop:stationary-rejection} give
		\[
		\begin{aligned}
			\norm{\overline r_t}_{L^p(\pi)}
			&\leq\frac2t\int_{t/2}^t\norm{r_h}_{L^p(\pi)}\,\df h\\
			&\leq CL\sqrt{p(d+p)}\,\frac2t\int_{t/2}^t h\,\df h
			=\frac{3C}{4}Lt\sqrt{p(d+p)}.
		\end{aligned}
		\]
		Set
		\[
		G_{t,p}=\{x:\overline r_t(x)\leq1/3\}.
		\]
		Markov's inequality yields
		\begin{equation*}
			\pi(G_{t,p}^c)
			\leq
			\left(\frac{9C}{4}Lt\sqrt{p(d+p)}\right)^p,
		\end{equation*}
		which is \eqref{eq:overlap-exceptional} after renaming the universal
		constant.
		
		We next compare nearby proposal laws.  Because $U$ is convex and has $L$-Lipschitz gradient, the Baillon--Haddad cocoercivity
		inequality states that \citep[Theorem~2.1.5]{Nesterov2004}
		\[
		\ip{\nabla U(x)-\nabla U(y)}{x-y}
		\geq
		\frac1L\norm{\nabla U(x)-\nabla U(y)}^2.
		\]
		See also \citet{BauschkeCombettes2010}.
		Consequently, for \(h\leq t\),
		\[
		\begin{aligned}
			&\norm{x-h\nabla U(x)-y+h\nabla U(y)}^2\\
			&\quad=\norm{x-y}^2
			-2h\ip{x-y}{\nabla U(x)-\nabla U(y)}
			+h^2\norm{\nabla U(x)-\nabla U(y)}^2\\
			&\quad\leq\norm{x-y}^2.
		\end{aligned}
		\]
		Note that we have used $h \leq t \leq 2/L$.
		Thus the proposal mean map $x \mapsto x - h \nabla U(x)$ is nonexpansive.  Two Gaussian laws with
		covariance \(2hI_d\) and means \(\mu_x,\mu_y\) have Kullback--Leibler
		divergence \(\norm{\mu_x-\mu_y}^2/(4h)\).  Pinsker's inequality
		\citep[Lemma~2.5]{Tsybakov2009} therefore gives
		\begin{equation}\label{eq:proposal-TV}
			\TV(Q_h(x,\cdot),Q_h(y,\cdot))
			\leq
			\frac{\norm{x-y}}{\sqrt{8h}}
			\leq
			\frac{\norm{x-y}}{2\sqrt{t}},
			\qquad h\in[t/2,t].
		\end{equation}
		For fixed \(x\), write
		\[
		R_{h,x}(\df y)=[1-\alpha_h(x,y)] \, Q_h(x,\df y).
		\]
		This measure has total mass \(r_h(x)\), and, because \(Q_h(x,\cdot)\) has
		a Lebesgue density, it gives no mass to \(\{x\}\).  Hence
		\[
		P_h(x,\cdot)-Q_h(x,\cdot)
		=r_h(x)\delta_x-R_{h,x}
		\]
		is the Jordan decomposition into mutually singular positive and negative
		parts.  This proves
		\begin{equation}\label{eq:mala-proposal-tv-exact}
			\TV(P_h(x,\cdot),Q_h(x,\cdot))=r_h(x).
		\end{equation}
		Total variation is convex under mixing: for probability kernels
		\(M_h,N_h\) and a probability measure \(\rho\),
		\[
		\TV\left(\int M_h\,\df\rho,\int N_h\,\df\rho\right)
		\leq\int\TV(M_h,N_h)\,\df\rho.
		\]
		Applying this fact three times, together with the triangle inequality,
		\eqref{eq:mala-proposal-tv-exact}, and \eqref{eq:proposal-TV}, gives, for
		\(x,y\in G_{t,p}\) with \(\norm{x-y}\leq\sqrt{t}/16\),
		\begin{align*}
			\TV(\cK_t(x,\cdot),\cK_t(y,\cdot))
			&\leq\overline r_t(x)
			+\frac2t\int_{t/2}^t
			\TV(Q_h(x,\cdot),Q_h(y,\cdot))\,\df h
			+\overline r_t(y)\\
			&\leq\frac13+\frac1{32}+\frac13
			=\frac{67}{96}<\frac34.
		\end{align*}
		This proves the first assertion of \cref{prop:overlap}.
		We note that this type of argument is common in the MCMC literature; see, e.g., the proof of Theorem~3 in \cite{WuSchmidlerChen2022}.
		
		It remains to prove the final assertion.
		Let $Z \sim N_d(0, I_d)$, and \(Y=x-h \nabla U(x) +\sqrt{2h}Z\).  We first derive the log-ratio identity used below.
		Since, for $a, b \in \mathbb{R}^d$,
		\[
		q_h(a,b)=(4\pi h)^{-d/2}
		\exp\left[-\frac{\norm{b-a+h\nabla U(a)}^2}{4h}\right],
		\]
		one has
		\begin{align*}
			\log\frac{e^{-U(Y)}q_h(Y,x)}{e^{-U(x)}q_h(x,Y)}
			& =U(x)-U(Y)\\
			&\quad+\frac1{4h}
			\left(\norm{Y-x+h \nabla U(x)}^2-\norm{x-Y+h \nabla U(Y)}^2\right).
		\end{align*}
		Expanding the two squared norms and collecting
		the gradient and Gaussian terms gives
		\begin{equation}\label{eq:MH-ratio-expansion}
			\begin{aligned}
				\log\frac{e^{-U(Y)}q_h(Y,x)}{e^{-U(x)}q_h(x,Y)}
				={}&U(x)-U(Y)-\ip{\nabla U(Y)}{x-Y}
				-\frac h4\norm{\nabla U(Y)-\nabla U(x)}^2\\
				&-\sqrt{\frac h2}\ip{Z}{\nabla U(Y)-\nabla U(x)}.
			\end{aligned}
		\end{equation}
		The Baillon--Haddad inequality
		\citep[Theorem~2.1.5]{Nesterov2004} gives
		\[
		U(x)-U(Y)-\ip{\nabla U(Y)}{x-Y}
		\geq
		\frac1{2L}\norm{\nabla U(Y)-\nabla U(x)}^2.
		\]
		If \(h\leq1/L\), the right-hand side of
		\eqref{eq:MH-ratio-expansion} is therefore at least
		\[
		\frac1{4L}\norm{\nabla U(Y)-\nabla U(x)}^2
		-\sqrt{\frac h2}\norm{Z}\norm{\nabla U(Y)-\nabla U(x)}
		\geq
		-\frac{Lh}{2}\norm{Z}^2.
		\]
		It follows that
		\[
		\alpha_h(x,Y)\geq e^{-Lh\norm{Z}^2/2},
		\qquad
		\E[\alpha_h(x,Y)]
		\geq(1+Lh)^{-d/2}
		\geq e^{-Lhd/2}.
		\]
		Let \(t'\in(0,1/(2Ld)]\).
		Then every \(h\in[t'/2,t']\) has rejection probability
		at most \(1-e^{-1/4}\), uniformly in \(x\).  Combining two such rejection
		bounds with \eqref{eq:proposal-TV}, we obtain, whenever
		$\norm{x-y}\leq\sqrt{t'}/16$,
		\[
		\TV(\cK_{t'}(x,\cdot),\cK_{t'}(y,\cdot))
		\leq
		2(1-e^{-1/4})+\frac1{32}<\frac34.
		\]
		This proves the final assertion of \cref{prop:overlap} and completes the proof.
	\end{proof}
	
	\section{Proof of \cref{prop:separated}}\label{app:gaussian}
	
	We repeatedly use the following notation for the standard Gaussian
	distribution:
	\[
	\varphi(x)=\frac{1}{\sqrt{2\pi}}e^{-x^2/2},
	\qquad
	\Phi(x)=\int_{-\infty}^x\varphi(v)\,\df v,
	\qquad
	\overline\Phi(x)=1-\Phi(x).
	\]

	\subsection{Some basic properties of Gaussian distribution functions}
	
	The Gaussian enlargement inequality \eqref{eq:bakry-ledoux-enlargement},
	justified above under the first-order assumptions, states that
	\[
	\pi(A^r)
	\geq
	\Phi\bigl(\Phi^{-1}(\pi(A))+\sqrt{m}\,r\bigr).
	\]
	We analyze the map $(q,s)\mapsto\Phi(\Phi^{-1}(q)+s)$.
	
	We begin with a Mills-ratio bound.
	
	\begin{lemma}
		Let \(a\geq0\).
		Then
		\begin{equation}\label{eq:mills-two-sided}
			\frac{\varphi(a)}{1+a}
			\leq
			\overline\Phi(a)
			\leq
			\frac{2\varphi(a)}{1+a}.
		\end{equation}
		Moreover,
		\begin{equation}\label{eq:log-tail-simple}
			\log\frac1{\overline\Phi(a)}
			\leq
			(1+a)^2.
		\end{equation}
	\end{lemma}
	
	\begin{proof}
		Let
		\[
		g(a)=\overline\Phi(a)-\frac{\varphi(a)}{1+a}.
		\]
		Then \(g(a)\to0\) as \(a\to\infty\), and direct differentiation gives
		\[
		g'(a)=-\frac{a\varphi(a)}{(1+a)^2}\leq0.
		\]
		Thus \(g(a)\geq0\), proving the lower bound in
		\eqref{eq:mills-two-sided}.  Similarly, for
		\[
		f(a)=\frac{2\varphi(a)}{1+a}-\overline\Phi(a),
		\]
		we have \(f(a)\to0\) as $a \to \infty$ and
		\[
		f'(a)=-\frac{(1+a^2)\varphi(a)}{(1+a)^2}<0.
		\]
		Hence \(f(a)\geq0\), proving the upper bound.  Classical sharper Mills
		bounds for \(a>0\) go back at least to \citet{Gordon1941}.
		
		The lower bound in \eqref{eq:mills-two-sided} implies
		\[
		\log\frac1{\overline\Phi(a)}
		\leq
		\frac{a^2}{2}+\log(1+a)+\frac12\log(2\pi).
		\]
		The difference between \((1+a)^2\) and the right-hand side is positive at
		\(a=0\), and its derivative is
		\[
		2+a-\frac1{1+a}>0.
		\]
		This proves \eqref{eq:log-tail-simple}.
	\end{proof}
	
	\begin{lemma} \label{lem:gaussian-shift}
		Let \(q\in(0,1/2]\) and \(s\geq0\).
		Then
		\begin{equation*}
			\Phi\bigl(\Phi^{-1}(q)+s\bigr)-q
			\geq
			\frac q4
			\min\left\{1,s\sqrt{\log(1/q)}\right\}.
		\end{equation*}
	\end{lemma}
	
	\begin{proof}
		Write \(\Phi^{-1}(q)=-a\), so \(a\geq0\) and
		\(q=\overline\Phi(a)\).  Then
		\[
		\Phi(-a+s)-\Phi(-a)
		=\int_0^s\varphi(a-v)\,\df v.
		\]
		For \(0\leq v\leq\min\{s,1\}\),
		\[
		\frac{\varphi(a-v)}{\varphi(a)}
		=e^{av-v^2/2}\geq e^{-1/2}.
		\]
		By the upper Mills bound in \eqref{eq:mills-two-sided},
		\(\varphi(a)\geq q(1+a)/2\).  Hence
		\[
		\begin{aligned}
			\Phi(-a+s)-\Phi(-a)
			&\geq
			\frac{e^{-1/2}}2q(1+a)\min\{s,1\}\\
			&\geq
			\frac{e^{-1/2}}2q
			\min\{1,s(1+a)\}\\
			&\geq
			\frac q4
			\min\left\{1,s\sqrt{\log(1/q)}\right\}.
		\end{aligned}
		\]
		The second line uses
		\((1+a)\min\{s,1\}\geq\min\{1,s(1+a)\}\); the third uses
		\eqref{eq:log-tail-simple} and \(e^{-1/2}/2>1/4\).
	\end{proof}
	
	
	\subsection{Proving \cref{prop:separated}}
	
	\begin{proof}
		If \(r=0\), the right-hand side of \eqref{eq:separated} is zero, so the
		conclusion is immediate.  Assume \(r>0\) and,
		without loss of generality, that \(\pi(A)=q\).  Since
		\(\operatorname{dist}(A,B)\geq r\), the open enlargement \(A^r\) is
		disjoint from \(B\), and hence \(A^r\subseteq A\cup C\).  Combining
		\eqref{eq:bakry-ledoux-enlargement} and \cref{lem:gaussian-shift} gives
		\[
		\begin{aligned}
			\pi(C)
			&\geq\pi(A^r)-\pi(A)\\
			&\geq\Phi\bigl(\Phi^{-1}(q)+\sqrt{m}\,r\bigr)-q\\
			&\geq\frac q4
			\min\left\{1,r\sqrt{m\log(1/q)}\right\}.
		\end{aligned}
		\]
		This is exactly \eqref{eq:separated}.
	\end{proof}
	
	\section{Proof of \cref{prop:flow}}\label{app:defective}

	\subsection{A generic conductance bound}
	
	We begin with a generic one-step flow bound for kernels that satisfy a
	close-coupling condition under the Gaussian isoperimetric inequality in
	\cref{prop:separated}.  Combining close coupling with an isoperimetric
	inequality to lower-bound one-step flows, and hence conductance, is a standard
	technique.  Related defective- or truncated-conductance arguments appear in
	\citet{DwivediEtAl2019,WuSchmidlerChen2022,ChenGatmiry2023,LiuTong2026}.
	
	\begin{lemma}\label{lem:defective}
		Let \(K\) be \(\pi\)-reversible, let \(t>0\), and let
		\(G\subseteq\R^d\) satisfy
		\[
		x,y\in G,\qquad \norm{x-y}\leq\frac{\sqrt{t}}{16}
		\quad\Longrightarrow\quad
		\TV(K(x,\cdot),K(y,\cdot))\leq\frac34.
		\]
		Let \(S\) be a measurable set such that \(0<\pi(S)\leq1/2\), and suppose that
		\begin{equation}\label{eq:defective-budget}
			\pi(G^c)
			\leq
			 2^{-12} \pi(S)
			\min\left\{1,\sqrt{mt\log[1/\pi(S)]}\right\}.
		\end{equation}
		Then
		\begin{equation}\label{eq:defective-flow}
			\mathcal{J}_K(S,S^c)
			\geq
			 2^{-12} \pi(S)
			\min\left\{1,\sqrt{mt\log[1/\pi(S)]}\right\}.
		\end{equation}
	\end{lemma}
	
	\begin{proof}
		Fix a measurable \(S\) with \(0<\pi(S)\leq1/2\).  Let
		\[
		\beta=\min\left\{1,\sqrt{mt\log[1/\pi(S)]}\right\}.
		\]
		Suppose, for contradiction, that
		\begin{equation} \label{ine:assumed_J}
			\mathcal{J}_K(S,S^c)< 2^{-12}  \beta \pi(S).
		\end{equation}
		Define
		\[
		\begin{aligned}
			S_1&=\{x\in S\cap G:K(x,S^c)< 1/8 \},\\
			S_2&=\{y\in S^c\cap G:K(y,S)< 1/8 \}.
		\end{aligned}
		\]
		For \(x\in S_1\) and \(y\in S_2\),
		\[
		\TV(K(x,\cdot),K(y,\cdot))
		\geq K(x,S)-K(y,S)> \frac34 .
		\]
		The close-coupling condition therefore implies
		\begin{equation}\label{eq:S1S2-distance}
			\operatorname{dist}(S_1,S_2)\geq\frac{\sqrt{t}}{16}.
		\end{equation}
		
		Markov's inequality and reversibility give
		\begin{equation} \label{ine:Markov}
			\pi(S_1^c \cap S \cap G) \leq 8 \mathcal{J}_K(S,S^c), \quad \pi(S_2^c \cap S^c \cap G) \leq 8 \mathcal{J}_K(S,S^c).
		\end{equation}
		Thus,
		\[
		\begin{aligned}
			\pi(S_1)
			&\geq \pi(S)-\pi(S\cap G^c)-8\mathcal{J}_K(S,S^c),\\
			\pi(S_2)
			&\geq1-\pi(S)-\pi(S^c\cap G^c)-8\mathcal{J}_K(S,S^c).
		\end{aligned}
		\]
		By \eqref{eq:defective-budget} and the assumed upper bound \eqref{ine:assumed_J},
		\[
		\pi(S_1)
		>\pi(S)- 9\cdot2^{-12}  \beta \pi(S),
		\qquad
		\pi(S_2)
		>1-\pi(S)- 9\cdot2^{-12}  \beta \pi(S).
		\]
		Since \(0<\beta\leq1\), \( 9\cdot2^{-12} <1/2\), and
		\(1-\pi(S) \geq \pi(S)\), it follows that
		\begin{equation}\label{eq:S1S2-masses}
			\pi(S_1)\geq\frac{\pi(S)}{2},
			\qquad
			\pi(S_2)\geq\frac{\pi(S)}{2}.
		\end{equation}
		Set
		\(q'=\min\{\pi(S_1),\pi(S_2)\}\).  Because \(S_1\subseteq S\),
		\[
		\frac{\pi(S)}{2}\leq q'\leq \pi(S) \leq\frac12.
		\]
		Apply \cref{prop:separated} to \(S_1,S_2\), and their complement.  Since
		\(q'\leq \pi(S)\), one has \(\log(1/q')\geq\log[1/\pi(S)]\).  Using
		\eqref{eq:S1S2-distance}, \eqref{eq:S1S2-masses}, and
		\(\min\{1,x/16\}\geq\min\{1,x\}/16\), we obtain
		\begin{align}
			\pi((S_1\cup S_2)^c)
			&\geq\frac{q'}4
			\min\left\{1,\frac1{16}\sqrt{mt\log(1/q')}\right\}\notag\\
			&\geq\frac{\pi(S)}{8}\cdot\frac1{16}
			\min\left\{1,\sqrt{mt\log[1/\pi(S)]}\right\}
			=\frac{\beta \pi(S)}{128}.
			\label{eq:defective-lower-complement}
		\end{align}
		
		On the other hand, by \eqref{ine:Markov}, \eqref{eq:defective-budget}, and \eqref{ine:assumed_J},
		\[
		\begin{aligned}
			\pi((S_1\cup S_2)^c) &= \pi(G^c) + \pi(G \cap S \cap S_1^c) + \pi(G \cap S^c \cap S_2^c) \\
			&\leq
			\pi(G^c)+ 16 \mathcal{J}_K(S,S^c)\\
			&< 17\cdot2^{-12}  \beta \pi(S)
			<\frac{\beta \pi(S)}{128},
		\end{aligned}
		\]
		contradicting \eqref{eq:defective-lower-complement}.  Therefore the assumed
		upper bound \eqref{ine:assumed_J} is impossible, and \eqref{eq:defective-flow}
		holds.
	\end{proof}

	\subsection{A technical lemma}
	
	The proof of \cref{prop:flow} via \cref{lem:defective} requires one additional
	technical lemma.
	
	\begin{lemma} \label{lem:exceptional-budget}
		Let \(c_{\mathrm r},C_{\mathrm r}\) be the constants in
		\cref{prop:overlap}.  Fix a number \(b_0>0\) satisfying
		\begin{equation} \label{eq:b0-choice}
			b_0\leq
			\min\left\{
			\frac12,
			c_{\mathrm r},
			\frac{1}{16C_{\mathrm r}}
			\right\}.
		\end{equation}
		Let \(p\geq  \widetilde p_\star:=A_0p_\star \),
		\(\theta\in(0,1]\), and
		\[
		t=\frac{\theta b_0}{L\sqrt{p(d+p)}},
		\]
		where $A_0$ satisfies
		\begin{equation}\label{eq:explicit-A0-choice}
			A_0\geq
			\frac{ (49/4) \log 2+(1/2)\log(2/b_0)}
			{\log 16-(1/2)}.
		\end{equation}
		Then, for every \(u\in[\log 2,p/2]\),
		\begin{equation}\label{eq:exceptional-budget-appendix}
			\left(\frac{\theta}{16}\right)^p
			\leq
			 2^{-12} e^{-u}\sqrt{m t u}.
		\end{equation}
		Moreover, \(mtu\leq b_0/2<1\) throughout this range.
	\end{lemma}
	
	\begin{proof}
		Since \(b_0\leq1/2\), the bound
		\eqref{eq:explicit-A0-choice} implies \(A_0\geq2\).
		Consequently, since \(p_\star\geq1\),
		\[
		p\geq\widetilde p_\star\geq A_0\geq2>2\log2>1,
		\]
		so \([\log2,p/2]\) is nonempty.
		The last assertion follows directly from \(m/L=1/\kappa\):
		\[
		mtu
		=\frac{\theta b_0u}{\kappa\sqrt{p(d+p)}}
		\leq
		\frac{\theta b_0}{2\kappa}
		\sqrt{\frac{p}{d+p}}
		\leq\frac{b_0}{2}.
		\]
		We focus on proving \eqref{eq:exceptional-budget-appendix}.
		Since \(u\mapsto e^{-u}\sqrt{u}\) decreases on \([1/2,\infty)\), the
		right-hand side of \eqref{eq:exceptional-budget-appendix} is smallest at
		\(u=p/2\).  At $u=p/2$ it equals
		\[
		 2^{-12} e^{-p/2}
		\left(\frac{\theta b_0}{2\kappa}\right)^{1/2}
		\left(\frac{p}{d+p}\right)^{1/4}.
		\]
		Dividing the desired inequality at \(u=p/2\) by
		\(\theta^{1/2}\) shows that its left-hand side contains the factor
		\(\theta^{p-1/2}\).  Since \(p\geq1\) and \(0<\theta\leq1\), this
		factor is at most one.  It is therefore sufficient that
		\begin{equation}\label{eq:A0-sufficient}
			p\left(\log 16-\frac12\right)
			\geq
			12 \log 2
			+\frac12\log\frac{2\kappa}{b_0}
			+\frac14\log\frac{d+p}{p}.
		\end{equation}
		Since \(p\geq1\) and \(d\geq1\),
		\[
		\log\frac{d+p}{p}\leq\log(2d)=\log d+\log2.
		\]
		The right-hand side of \eqref{eq:A0-sufficient} is therefore at most
		\[
		 \frac{49}{4} \log 2+\frac12\log\frac{2}{b_0}
		+\frac12\log\kappa+\frac14\log d.
		\]
		By \eqref{eq:explicit-A0-choice} and \(b_0\leq1/2\),
		\[
		A_0\left(\log16-\frac12\right)
		\geq \frac{49}{4} \log2+\frac12\log \frac{2}{b_0}
		\geq \frac{53}{4} \log2>\frac12.
		\]
		Since \(\log d\) and \(\log\kappa\) are nonnegative, these
		bounds and \(p\geq  \widetilde p_\star \) give
		\begin{align*}
			p\left(\log16-\frac12\right)
			&\geq A_0\left(\log16-\frac12\right)
			(1+\log d+\log\kappa)\\
			&\geq \frac{49}{4} \log2+\frac12\log \frac{2}{b_0}
			+\frac14\log d+\frac12\log\kappa.
		\end{align*}
		This proves \eqref{eq:A0-sufficient}, uniformly in all parameters.
	\end{proof}
	
	\subsection{Proving Proposition~\ref{prop:flow}}
	
	\begin{proof}
		We begin with the first assertion.
		Let $A_0$ and $b_0$ be universal constants satisfying
		\eqref{eq:b0-choice} and \eqref{eq:explicit-A0-choice}.
		Let \(p\geq  \widetilde p_\star:=A_0p_\star \),
		\(\theta\in(0,1]\), and
		\[
		t=\frac{\theta b_0}{L\sqrt{p(d+p)}}.
		\]
		Let $S\subseteq\R^d$ be measurable with
		$e^{-p/2}\leq\pi(S)\leq1/2$.
		By \cref{lem:exceptional-budget}, since
		$\log[1/\pi(S)]\in[\log 2,p/2]$,
		\[
		\left(\frac{\theta}{16}\right)^p
		\leq 2^{-12} \pi(S)\sqrt{mt\log[1/\pi(S)]}
		= 2^{-12} \pi(S)\min\left\{1,\sqrt{mt\log[1/\pi(S)]}\right\}.
		\]
		The equality follows from the final assertion of
		\cref{lem:exceptional-budget}, namely,
		$mt\log[1/\pi(S)]\leq1$.
		Because $b_0\leq c_{\mathrm r}$ and
		$b_0\leq1/(16C_{\mathrm r})$, the first assertion of
		\cref{prop:overlap} yields a measurable set $G_{t,p}\subseteq\R^d$ such
		that
		\[
		\pi(G_{t,p}^c) \leq \left[ C_{\mathrm{r}} L \frac{\theta b_0}{L\sqrt{p(d+p)}} \sqrt{p(d+p)} \right]^p \leq \left( \frac{\theta}{16}\right)^p,
		\]
		and
		\[
		x,y\in G_{t,p},\qquad
		\norm{x-y}\leq\frac{\sqrt{t}}{16}
		\quad\Longrightarrow\quad
		\TV(\cK_t(x,\cdot),\cK_t(y,\cdot))\leq\frac34.
		\]
		Then, by \cref{lem:defective},
		\[
		\mathcal{J}_{\cK_t}(S,S^c)
		\geq 2^{-12} \pi(S)\min\left\{1,\sqrt{mt\log[1/\pi(S)]}\right\}
		= 2^{-12} \pi(S)\sqrt{mt\log[1/\pi(S)]}.
		\]
		This proves the first assertion.
		
		For the second assertion, let $t'\in(0,1/(2Ld)]$ and let
		$S\subseteq\R^d$ be measurable with $0<\pi(S)\leq1/2$.  By the second
		assertion of \cref{prop:overlap},
		\[
		x,y\in\R^d,\quad
		\norm{x-y}\leq\frac{\sqrt{t'}}{16}
		\quad\Longrightarrow\quad
		\TV(\cK_{t'}(x,\cdot),\cK_{t'}(y,\cdot))\leq\frac34.
		\]
		Applying \cref{lem:defective} with $G=\R^d$ yields
		\[
		\mathcal{J}_{\cK_{t'}}(S,S^c)
		\geq 2^{-12} \pi(S)
		\min\left\{1,\sqrt{mt'\log[1/\pi(S)]}\right\}.
		\]
		This proves the second assertion.
	\end{proof}

	\section{Proof of Lemma~\ref{lem:fractional}}\label{app:fractional}
	
	The result can also be derived from the general Cheeger inequality of
	\citet{ChenWang2000}.  The proof below is direct and self-contained; we
	include the details because the precise dependence on \(\beta_j\) and
	\(\gamma_j\) is needed later.
	
	\begin{proof}
		For \(j=1,\ldots,N\), define the symmetric measure
		\[
		\mathsf J_j(\df x,\df y)
		=
		\varpi(\df x)K_j(x,\df y).
		\]
		The measure \(\mathsf J_j\) is symmetric because \(K_j\) is
		\(\varpi\)-reversible.  Both of its marginals equal \(\varpi\).
		
		We first prove 
		\begin{equation} \label{eq:gap-fractional-pre}
			\sum_{j=1}^N
			\gamma_j\cE_{K_j}(g,g)
			\geq
			\frac{
				\norm{g}_{L^2(\varpi)}^2
			}{
				2\sum_{j=1}^N \beta_j^2/\gamma_j
			}
		\end{equation}
		for a nonnegative function \(g\in L^2(\varpi)\)
		such that
		$\varpi(g>0)\leq 1/2$.
		For \(t\geq0\), let
		\[
		A_t=\{x:g(x)^2>t\}.
		\]
		Then \(\varpi(A_t)\leq1/2\), so the assumed fractional boundary inequality
		gives
		\[
		\varpi(A_t)
		\leq
		\sum_{j=1}^N
		\beta_j \mathcal{J}_{K_j}(A_t,A_t^c).
		\]
		Integrating this inequality over \(t\geq0\) and using the layer-cake
		identity yields
		\[
		\begin{aligned}
			\norm{g}_{L^2(\varpi)}^2
			&=
			\int_0^\infty \varpi(A_t)\,\df t\\
			&\leq
			\sum_{j=1}^N
			\beta_j
			\int_0^\infty
			\mathcal{J}_{K_j}(A_t,A_t^c)\,\df t.
		\end{aligned}
		\]
		
		For each \(j\), symmetry of \(\mathsf J_j\) gives the coarea identity
		\[
		\int_0^\infty
		\mathcal{J}_{K_j}(A_t,A_t^c)\,\df t
		=
		\frac12
		\iint
		\abs{g(x)^2-g(y)^2}
		\,\mathsf J_j(\df x,\df y).
		\]
		Indeed, for fixed \(x,y\),
		\[
		\int_0^\infty
		\1_{\{g(x)^2>t\geq g(y)^2\}}\,\df t
		=
		\bigl[g(x)^2-g(y)^2\bigr]_+,
		\]
		and symmetry changes the integral of the positive part into one half
		of the integral of the absolute difference.  Consequently,
		\[
		\norm{g}_{L^2(\varpi)}^2
		\leq
		\frac12
		\sum_{j=1}^N
		\beta_j
		\iint
		\abs{g(x)-g(y)}\,\abs{g(x)+g(y)}
		\,\mathsf J_j(\df x,\df y).
		\]
		
		Apply Cauchy--Schwarz simultaneously over the component index \(j\)
		and the pair \((x,y)\):
		\[
		\begin{aligned}
			\norm{g}_{L^2(\varpi)}^2
			\leq{}&
			\frac12
			\left\{
			\sum_{j=1}^N
			\gamma_j
			\iint
			[g(x)-g(y)]^2
			\,\mathsf J_j(\df x,\df y)
			\right\}^{1/2}\\
			&\times
			\left\{
			\sum_{j=1}^N
			\frac{\beta_j^2}{\gamma_j}
			\iint
			[g(x)+g(y)]^2
			\,\mathsf J_j(\df x,\df y)
			\right\}^{1/2}.
		\end{aligned}
		\]
		The first square root equals
		\[
		\left[
		2\sum_{j=1}^N
		\gamma_j\cE_{K_j}(g,g)
		\right]^{1/2}.
		\]
		For the second square root, the inequality
		\[
		[g(x)+g(y)]^2
		\leq
		2g(x)^2+2g(y)^2
		\]
		and the fact that both marginals of \(\mathsf J_j\) equal \(\varpi\)
		give
		\[
		\iint
		[g(x)+g(y)]^2
		\,\mathsf J_j(\df x,\df y)
		\leq
		4\norm{g}_{L^2(\varpi)}^2.
		\]
		It follows that
		\[
		\norm{g}_{L^2(\varpi)}^2
		\leq
		\left[
		2\sum_{j=1}^N
		\gamma_j\cE_{K_j}(g,g)
		\right]^{1/2}
		\norm{g}_{L^2(\varpi)}
		\left(
		\sum_{j=1}^N\frac{\beta_j^2}{\gamma_j}
		\right)^{1/2}.
		\]
		If \(g=0\) in \(L^2(\varpi)\), there is nothing to prove.  Otherwise,
		cancelling one factor of \(\norm{g}_{L^2(\varpi)}\) and squaring
		gives \eqref{eq:gap-fractional-pre}.

		Now let \(f\in L^2(\varpi)\), and let \(c\) be a median of \(f\); thus
		\[
		\varpi(f\geq c)\geq\frac12,
		\qquad
		\varpi(f\leq c)\geq\frac12.
		\]
		Define the positive and negative parts of $f-c$:
		\[
		g_+=(f-c)_+,
		\qquad
		g_-=(f-c)_-.
		\]
		Then
		\[
		\varpi(g_+>0)\leq\frac12,
		\qquad
		\varpi(g_->0)\leq\frac12.
		\]
		Moreover, for every \(x,y\),
		\[
		\begin{aligned}
			[f(x)-f(y)]^2
			\geq{}&
			\bigl[g_+(x)-g_+(y)\bigr]^2
			+
			\bigl[g_-(x)-g_-(y)\bigr]^2.
		\end{aligned}
		\]
		To see this, write
		\[
		f-c=g_+-g_-.
		\]
		The two differences
		\[
		g_+(x)-g_+(y)
		\quad\text{and}\quad
		g_-(x)-g_-(y)
		\]
		have opposite signs, so the cross term in the square of their
		difference is nonnegative.
		
		Therefore, by \eqref{eq:gap-fractional-pre},
		\[
		\begin{aligned}
			\cE_P(f,f)
			&\geq
			\sum_{j=1}^N
			\gamma_j\cE_{K_j}(f,f)\\
			&\geq
			\sum_{j=1}^N
			\gamma_j\cE_{K_j}(g_+,g_+)
			+
			\sum_{j=1}^N
			\gamma_j\cE_{K_j}(g_-,g_-)\\
			&\geq
			\frac{
				\norm{g_+}_{L^2(\varpi)}^2
				+
				\norm{g_-}_{L^2(\varpi)}^2
			}{
				2\sum_{j=1}^N \beta_j^2/\gamma_j
			}.
		\end{aligned}
		\]
		Since \(g_+\) and \(g_-\) have disjoint supports and \(f-c=g_+-g_-\),
		\[
		\norm{g_+}_{L^2(\varpi)}^2
		+
		\norm{g_-}_{L^2(\varpi)}^2
		=
		\norm{f-c}_{L^2(\varpi)}^2.
		\]
		The mean minimizes \(a\mapsto\norm{f-a}_{L^2(\varpi)}^2\), so
		\[
		\norm{f-c}_{L^2(\varpi)}^2
		\geq
		\Var_\varpi(f).
		\]
		We conclude that
		\[
		\cE_P(f,f)
		\geq
		\frac{
			\Var_\varpi(f)
		}{
			2\sum_{j=1}^N \beta_j^2/\gamma_j
		}.
		\]
		Taking the infimum over all nonconstant \(f\in L^2(\varpi)\) proves the
		claim.
	\end{proof}
	
	\section{Multiscale arithmetic}
	
	This appendix verifies a summation estimate used in
	\cref{sec:main-proof}.  The constants \(b_0 \leq 1/2\) and \(A_0 \geq 1\) are the universal
	choices made at the beginning of the proof of \cref{thm:main} in \cref{sec:main-proof}, and $\widetilde p_\star=A_0p_\star$.

	\begin{lemma} \label{lem:ladder-sum}
		Assume \( \widetilde p_\star <d\), and define \(p_j,J,t_j,\theta\) by
		\eqref{eq:pj}--\eqref{eq:theta-tj}.  With
		\[
		\gamma_j=\frac{t_j}{2H},
		\qquad
		\phi_0^2= 2^{-24} m t_0\log 2,
		\qquad
		\phi_j^2= 2^{-27} m t_jp_j\quad(j\geq1),
		\]
		there is a universal \(C \in (0,\infty)\) such that
		\[
		\sum_{j=0}^J\frac1{\gamma_j\phi_j^2}
		\leq
		C\frac{HL^2 \widetilde p_\star (d+ \widetilde p_\star )}
		{m\theta^2b_0^2}.
		\]
	\end{lemma}
	
	\begin{proof}
		For the first component,
		\[
		\frac1{\gamma_0\phi_0^2}
		=
		\frac{ 2^{25} H}{m t_0^2\log 2}
		=
		\frac{ 2^{25} HL^2 \widetilde p_\star (d+ \widetilde p_\star )}
		{m\theta^2b_0^2\log 2}.
		\]
		For \(j\geq1\),
		\[
		\frac1{\gamma_j\phi_j^2}
		=
		\frac{ 2^{28} H}{m t_j^2p_j}
		=
		\frac{ 2^{28} HL^2(d+p_j)}
		{m\theta^2b_0^2}.
		\]
		
		It remains to show
		\[
		\widetilde p_\star (d+ \widetilde p_\star ) + \sum_{j=1}^J (d+p_j) \leq C  \widetilde p_\star (d+ \widetilde p_\star )
		\]
		for a universal~$C$.
		Because \( \widetilde p_\star <d\), one
		has \(J\geq1\).  Minimality of \(J\) gives
		\(4^{J-1} \widetilde p_\star <d\), and hence
		\[
		J<1+\log_4\frac d{ \widetilde p_\star }
		\leq1+\log_4 d.
		\]
		The geometric-series formula and \eqref{eq:pJ-range} give
		\[
		\sum_{j=1}^Jp_j
		= \widetilde p_\star \sum_{j=1}^J4^j
		=\frac{4p_J-4 \widetilde p_\star }{3}
		<\frac{16d}{3}.
		\]
		Moreover, \( \widetilde p_\star \geq A_0(1+\log d)\),
		\(A_0\geq1\), and \( \widetilde p_\star \geq1\) imply
		\[
		1+\log_4 d
		\leq1+\frac{ \widetilde p_\star }{A_0\log4}
		\leq\left(1+\frac1{\log4}\right) \widetilde p_\star .
		\]
		Therefore
		\[
		\sum_{j=1}^J(d+p_j)
		=Jd+\sum_{j=1}^Jp_j
		\leq C  \widetilde p_\star  d
		\]
		for some universal constant $C \in (0,\infty)$.
		The desired inequality then follows by enlarging~$C$.
	\end{proof}

	\section{A spectral gap for small fixed step sizes}
	\label{app:small-fixed-step-gap}
	
	The uniform acceptance estimate used in the proof of
	\cref{prop:overlap} also yields a spectral-gap lower bound for
	fixed-step MALA. 
	We record this consequence below.
	The result is also described in \cref{rem:minimax-fixed-step-ceiling}.
	
	\begin{proposition}\label{prop:small-fixed-step-gap}
		Suppose that \(U\) satisfies
		\eqref{eq:first-order-assumptions}, and let \(\kappa=L/m\).
		There is a universal constant \(c\in(0,\infty)\) such that,
		for every
		\[
		0<h\leq\frac{1}{2Ld},
		\]
		the fixed-step MALA kernel satisfies
		\[
		\Gap(P_h)\geq cmh.
		\]
		In particular, choosing \(h=1/(2Ld)\) gives
		\[
		\Gap(P_h)\geq\frac{c}{\kappa d}.
		\]
	\end{proposition}
	
	\begin{proof}
		Fix \(h\in(0,1/(2Ld)]\). For \(x\in\R^d\), let
		\[
		Y=x-h\nabla U(x)+\sqrt{2h}\,Z,
		\qquad Z\sim N_d(0,I_d).
		\]
		Since \(h\leq1/L\), the argument following
		\eqref{eq:MH-ratio-expansion} in the proof of
		\cref{prop:overlap} gives
		\[
		\begin{aligned}
			1-r_h(x)
			&=\E[\alpha_h(x,Y)]\geq e^{-Lhd/2}
			\geq e^{-1/4}.
		\end{aligned}
		\]
		Thus \(r_h(x)\leq1-e^{-1/4}\) for every \(x\in\R^d\).
		
		Because \(U\) is convex and has an \(L\)-Lipschitz gradient,
		the proposal mean map \(x\mapsto x-h\nabla U(x)\) is
		nonexpansive whenever \(h\leq2/L\), as shown in the proof of
		\cref{prop:overlap}. The Gaussian comparison leading to
		\eqref{eq:proposal-TV} therefore gives
		\[
		\TV(Q_h(x,\cdot),Q_h(y,\cdot))
		\leq\frac{\norm{x-y}}{\sqrt{8h}}.
		\]
		Combining this inequality with
		\eqref{eq:mala-proposal-tv-exact} and the triangle inequality,
		we find that, whenever \(\norm{x-y}\leq\sqrt h/16\),
		\[
		\begin{aligned}
			\TV(P_h(x,\cdot),P_h(y,\cdot))
			&\leq r_h(x)
			+\TV(Q_h(x,\cdot),Q_h(y,\cdot))
			+r_h(y)\\
			&\leq2(1-e^{-1/4})+\frac{1}{16\sqrt8}
			<\frac34.
		\end{aligned}
		\]
		This overlap bound holds for all \(x,y\in\R^d\).
		
		Apply \cref{lem:defective} with \(K=P_h\), \(t=h\),
		and \(G=\R^d\). Its exceptional-set condition holds because
		\(\pi(G^c)=0\). Consequently, for every measurable \(S\)
		such that \(0<\pi(S)\leq1/2\),
		\[
		\mathcal{J}_{P_h}(S,S^c)
		\geq
		\pRevision{2^{-12}}\pi(S)
		\min\left\{1,\sqrt{mh\log[1/\pi(S)]}\right\}.
		\]
		Since
		\[
		\log[1/\pi(S)]\geq\log2,
		\qquad
		mh\log2\leq\frac{\log2}{2\kappa d}<1,
		\]
		it follows that
		\[
		\mathcal{J}_{P_h}(S,S^c)
		\geq
		 2^{-12} \sqrt{mh\log2}\,\pi(S).
		\]
		
		The one-kernel case of \cref{lem:fractional}, which is the
		usual Cheeger inequality \citep{LawlerSokal1988}, now gives
		\[
		\Gap(P_h)
		\geq
		\frac12\left( 2^{-12} \sqrt{mh\log2}\right)^2
		=
		 2^{-25} (\log2)\,mh.
		\]
		More explicitly, the lemma applies with
		\(P=K_1=P_h\), \(\varpi=\pi\), \(\gamma_1=1\), and
		\(\beta_1= 2^{12} /\sqrt{mh\log2}\).
		
		For \(h=1/(2Ld)\), the preceding bound becomes
		\[
		\Gap(P_h)
		\geq
		\frac{ 2^{-26} \log2}{\kappa d}.
		\]
		Both assertions therefore hold with
		\(c= 2^{-26} \log2\).
	\end{proof}

	\pdfbookmark[1]{References}{sec:references}

	\bibliographystyle{plainnat}
	\bibliography{uniform_random_mala}

\end{document}